\documentclass[12pt]{article}
\usepackage{amsmath}
\usepackage{amssymb}
\usepackage{amsthm}
\usepackage[totalwidth=17cm, totalheight=25cm]{geometry}

\newfont{\bigbf}{cmbx10 scaled\magstep1}
\normalbaselines
\title{Equivariant Ricci flow with surgery 
and applications\\
 to finite group actions on 
geometric 3-manifolds}
\author{Jonathan Dinkelbach, Bernhard Leeb\footnote{
dinkelbach@mathematik.uni-muenchen.de, b.l@lmu.de}}
\date{January 9, 2009}

\def\makeop#1{\expandafter\def\csname #1\endcsname{\mathop{\mathrm{#1}}\nolimits}}

\newtheoremstyle{boldplain}
  {9pt}
  {9pt}
  {\itshape}
  {}
  {\bfseries}
  {.}
  {.5em}
  {\thmname{#1}\thmnumber{ #2}\thmnote{ (#3)}}%

\newtheoremstyle{bolddefinition}
  {9pt}
  {9pt}
  {}
  {}
  {\bfseries}
  {.}
  {.5em}
  {\thmname{#1}\thmnumber{ #2}\thmnote{ (#3)}}%

\theoremstyle{boldplain}
\newtheorem{thm}[equation]{Theorem}
\newtheorem*{thmE}{Theorem E}
\newtheorem*{thmH}{Theorem H}
\newtheorem{prop}[equation]{Proposition}
\newtheorem{add}[equation]{Addendum}
\newtheorem{lem}[equation]{Lemma}
\newtheorem{sublem}[equation]{Sublemma}
\newtheorem{corollary}[equation]{Corollary}
\newtheorem{cor}[equation]{Corollary}
\newtheorem{rem}[equation]{Remark}

\theoremstyle{bolddefinition}
\newtheorem{defn}[equation]{Definition}

\numberwithin{equation}{section}

\def\co{\colon\thinspace}

\def\C{{\mathbb C}}
\def\R{{\mathbb R}}

\def\Z{{\mathbb Z}}

\def\al{\alpha}
\def\ga{\gamma}
\def\Ga{\Gamma}
\def\de{\delta}

\def\eps{\epsilon}
\def\la{\lambda}

\def\si{\sigma}
\def\Si{\Sigma}

\def\Om{\Omega}

\def\epsseven{\eps_1^{(1)}}
\def\epseight{\eps_1^{(2)}}
\def\epsnine{\eps_1^{(3)}}
\def\epsten{\eps_1^{(4)}}
\def\epseleven{\eps_1^{(5)}}

\def\Rm{R}
\def\scal{S}

\def\3{\ss}

\def\acts{\curvearrowright}
\def\D{\partial}
\def\embed{\hookrightarrow}
\def\half{\frac{1}{2}}
\def\lra{\longrightarrow}
\def\ol{\overline}

\makeop{Diff}
\makeop{Isom}
\makeop{Deck}
\makeop{diam}
\makeop{Stab}
\makeop{id}
\makeop{Fix}
\makeop{inj}

\DeclareMathOperator{\Int}{int}
\DeclareMathOperator{\Point}{pt}

\def\rrad{\raisebox{-0.3ex}{\ensuremath{\widetilde{\phantom{\mathrm{rad}}}}}
\hspace{-3.0ex}\mathrm{rad}}
\newcommand{\neck}{\mathrm{neck}}
\newcommand{\isom}{\mathrm{isom}}
\newcommand{\hyp}{\mathrm{hyp}}
\newcommand{\sph}{\mathrm{sph}}
\newcommand{\nn}{\mathrm{nn}}
\newcommand{\tildeB}{\widetilde{B}}

\def\BI{\begin{itemize}}
\def\EI{\end{itemize}}

\begin{document}
\maketitle

\begin{abstract}    
\noindent 
We apply an equivariant version of Perelman's Ricci flow with surgery 
to study smooth actions by finite groups on closed 3--manifolds.
Our main result is that such actions on elliptic and hyperbolic 3--manifolds 
are conjugate to isometric actions. 
Combining our results with results 
by Meeks and Scott \cite{MeeksScott}, 
it follows that 
such actions on geometric 3--manifolds (in the sense of Thurston)  
are always geometric,
i.e.\ there exist invariant locally homogeneous Riemannian metrics. 
This answers a question posed by Thurston in \cite{Thurston_Bull}. 
\end{abstract}


\section{Introduction}

The main results of this paper 
concern smooth group actions on geometric 3--manifolds: 

\begin{thmE}[Actions on elliptic manifolds are standard] 
Any smooth action by a finite group on an elliptic
3--manifold is smoothly conjugate to an isometric action. 
\end{thmE}

\begin{thmH}[Actions on closed hyperbolic manifolds are standard]
Any smooth action by a finite group on a closed hyperbolic 3--manifold
is smoothly conjugate to an isometric action.
\end{thmH}

We also show 
(Theorem~\ref{thm:s2s1})
that smooth actions by finite groups on closed $S^2\times\R$--manifolds 
are geometric, 
i.e.\ there exist invariant Riemannian metrics 
locally isometric to $S^2\times\R$. 
See Meeks and Yau \cite{MeeksYau} for earlier results concerning this case. 

Corresponding results for the other five 3--dimensional Thurston geometries
had been obtained by Meeks and Scott in \cite{MeeksScott}.
Combining our results with theirs it follows 
that smooth actions by finite groups on closed geometric 3--manifolds 
are always geometric, 
i.e.\ there exist invariant locally homogeneous Riemannian metrics. 
This answers a question posed by Thurston in \cite{Thurston_Bull}. 

Theorem H extends to actions 
on compact 3--manifolds with boundary 
whose interior admits a complete hyperbolic structure of finite volume 
(Theorem~\ref{cuspcase}). 

Theorems E and H 
have been previously known in most cases.
For free actions,
they are due to Perelman 
in his fundamental work on the Ricci flow in dimension three
\cite{Perelman_entropy,Perelman_surgery,Perelman_extinction}. 
For orientation preserving non-free actions,
they have been proven by Boileau, Leeb and Porti \cite{orbi}
along the lines suggested by Thurston \cite{Thurston_symm} and 
based on Thurston's Hyperbolization Theorem for Haken manifolds,
using techniques
from 3--dimensional topology, 
the deformation theory of geometric structures and 
the theory of metric spaces with curvature bounded below. 
They have also been known in several cases 
for orientation non-preserving non-free actions. 

In this paper we give a unified approach 
by applying the Ricci flow techniques to the case of non-free actions. 
This also provides an alternative proof in the orientation preserving 
non-free case.  
Our argument is based on several deep recent results 
concerning the Ricci flow with cutoff on closed 3--manifolds, 
namely its long time existence for arbitrary initial metrics
(Perelman \cite{Perelman_surgery}, Kleiner and Lott \cite{KL},
Morgan and Tian \cite{MorganTian_poinc}, Bamler \cite{Bamler}),
its extinction in finite time 
on non-aspherical prime 3--manifolds 
(Perelman \cite{Perelman_extinction}, Colding and Minicozzi
\cite{ColdingMinicozzi_ext, ColdingMinicozzi_width}, 
Morgan and Tian \cite{MorganTian_poinc})
and, for Theorem H, the analysis of its asymptotic long time behavior (Perelman
\cite{Perelman_surgery}, Kleiner and Lott \cite{KL}).

Given a smooth action $\rho\co G\acts M$ by a finite group 
on a closed 3--manifold,
and given any $\rho$--invariant initial Riemannian metric $g_0$ on $M$,
there is no problem in running an equivariant Ricci flow with cutoff, 
because the symmetries are preserved between surgery times
and can be preserved while performing the surgeries. 
During the flow,
the underlying manifold and the action may change. 
In order to understand the initial action $\rho$,
one needs to compare the actions before and after a surgery, 
and to verify that, 
short before the extinction of connected components, 
the actions on them are standard. 
The main issue to be addressed here 
is that the caps occurring in highly curved regions 
close to the singularities of the Ricci flow 
may have nontrivial stabilizers 
whose actions one has to control. 
Since on an elliptic 3--manifold the Ricci flow goes extinct in finite time 
for any initial metric, 
one can derive Theorem E. 
If $M$ is hyperbolic,
its hyperbolic metric is unique up to diffeomorphisms by Mostow rigidity. 
It turns out that for large times the time slice is (again) 
diffeomorphic to $M$
and the Riemannian metric produced by the Ricci flow 
converges smoothly to the hyperbolic metric 
up to scaling and diffeomorphisms. 
This leads to Theorem H. 
Our methods apply also to actions on the non-irreducible prime 3--manifold 
$S^2\times S^1$
and its quotient manifolds 
because the Ricci flow goes extinct in this case, too. 

{\em Acknowledgements.}
We are grateful to Richard Bamler for many useful comments 
on an earlier version of this article. 
We thank the 
Deutsche Forschungsgemeinschaft 
(DFG-project LE 1312/1) 
and the Hausdorff Research Institute for Mathematics in Bonn 
for financial support.

\tableofcontents

\section{Topological preliminaries} \label{sec:prelim}
\subsection{Standard actions on geometric 3--manifolds}
\label{defstan}

Throughout this paper, 
we consider smooth actions 
by finite groups on 3--manifolds.

As e.g.\ in Meeks and Scott \cite{MeeksScott}, 
we call an action $\rho\co G\acts M^3$ 
on a closed connected 3--manifold {\em standard}
if it preserves a geometric structure 
in the sense of Thurston \cite{Scott,Thurston3M},
i.e.\ if there exists an invariant 
locally homogeneous Riemannian metric. 
Note that this requires the manifold itself to be geometric 
and the type of geometric structure is uniquely determined. 

In the case of the elliptic and hyperbolic geometries
the geometric structure on $M$ is unique up to diffeomorphism. 
(In the elliptic case, 
this follows from the isometry classification of elliptic manifolds,
see e.g.\ Thurston \cite[Theorem 4.4.14]{Thurston3M},
and the topological classification of lens spaces by Brody \cite{Brody},
see also Hatcher \cite[Theorem 2.5]{Hatcher}.
In the hyperbolic case,
this is a consequence of Mostow rigidity \cite{Mostow}.) 
This means that an action 
on an elliptic or on a closed hyperbolic 3--manifold
is standard 
if and only if it is smoothly conjugate to an isometric action. 

Sometimes, we will need to consider actions 
on a few simple noncompact manifolds 
or compact manifolds with nonempty boundary: 
We call an action on the (open or closed) unit ball {\em standard},
if it is smoothly conjugate to an orthogonal action.  
We call an action on the round cylinder $S^2\times\R$, 
or on one of the orientable $\Z_2$--quotients 
$S^2\times_{\Z_2}\R\cong\R P^3-\bar B^3$ 
and $S^2\times_{\Z_2}[-1,1]\cong \R P^3-B^3$ 
{\em standard}, 
if it is smoothly conjugate to an isometric action
(say, with respect to the $S^2\times\R$--structure). 

If $M$ is disconnected, we call an action $G\acts M$ {\em standard}
if for each connected component $M_0$ of $M$ the restricted action 
$\Stab_G(M_0)\acts M_0$ is standard.

\subsection{Equivariant diffeomorphisms of the 2--sphere are standard}
\label{sec:2sph}

We need the following 
equivariant version of a result of Munkres \cite{Munkres}
regarding isotopy classes of diffeomorphisms of the 2--sphere. 

\begin{prop}
\label{equivdiffeo2sph}
Let $\rho\co H\acts S^2$ be an orthogonal action of a finite group 
on the 2--dimensional unit sphere. 
Then every $\rho$--equivariant diffeomorphism $S^2\to S^2$
is $\rho$--equivari\-ant\-ly isotopic to an isometric one.  
\end{prop}

Equivalently, in terms of quotient orbifolds:

\begin{prop}
\label{diffeo2orb} 
Any diffeomorphism $\phi\co \mathcal{O}_1\to \mathcal{O}_2$ 
of closed spherical 2--orbifolds is isotopic to an isometry.
\end{prop}

We recall that a diffeomorphism of orbifolds is a homeomorphism which
locally lifts to an equivariant diffeomorphism of orbifold charts.  
In particular,
it maps the singular locus to the singular locus and preserves the
types of singular points. 
\begin{proof}
For ${\mathcal O}_1\cong{\mathcal O}_2\cong S^2$
the result is proven in Munkres \cite{Munkres},
see also Thurston \cite[Theorem 3.10.11]{Thurston3M},
and for $\R P^{2}$ in Epstein \cite[Theorem 5.5]{Epstein_Curves}. 
We therefore assume that the orbifolds have singularities
and extend the result to this case 
using standard arguments from 2--dimensional topology. 

Let us first recall 
the list of non-smooth closed spherical 2--orbifolds.
By a cone point, we mean an isolated singular point. 
\begin{itemize}
\item The 2--sphere with two or three cone points,
and the projective plane with one cone point.
\item The closed 2--disk with reflector boundary, at most three corner
reflector points and at most one cone point in the interior.
(The cone point can only occur 
in the case of at most one corner reflector
and must occur if there is precisely one corner reflector.) 
\end{itemize}
From the classification we observe that 
a closed spherical 2--orbifold is determined up to isometry
by its underlying surface and the types of the singular points. 
Hence $\mathcal{O}_1$ and $\mathcal{O}_2$ must be isometric.
After suitably identifying them, 
we may regard $\phi$ as a self-diffeomorphism 
of a closed spherical 2--orbifold $\mathcal{O}$
which fixes every cone point and corner reflector. 
We can arrange moreover
that $\phi$ preserves orientation if $\mathcal{O}$ is orientable,
and locally preserves orientation near the cone point 
if $\mathcal{O}$ is a projective plane with one cone point. 
It is then a consequence that, 
when the singular locus is one-dimensional,
i.e.\ when ${\mathcal O}$ has reflector boundary, 
$\phi$ preserves every singular edge (and circle) and acts on it 
as an orientation preserving diffeomorphism. 

\begin{lem}
\label{isotoid_pt}
Let $D=D(1)$ be the open unit disk 
and let $D'=D(r')$, $0<r'\leq1$, be a round subdisk centered at the origin.
Suppose that $\rho\co H\acts D$ is an orthogonal action 
of a finite group, 
and that $\phi\co D'\to D$ is an orientation preserving $\rho$--equivariant 
smooth embedding fixing $0$. 
Then $\phi$ is isotopic, 
via a compactly supported $\rho$--equivariant isotopy,
to a $\rho$--equivariant smooth embedding $D'\to D$
which equals $\pm \id_{D'}$ near $0$ and, 
if $\rho$ preserves orientation,
$+\id_{D'}$.  
\end{lem}

\begin{proof}(Compare e.g.\ \cite[end of proof of Lemma 3.10.12]{Thurston3M}).
We may first isotope $\phi$ to make it linear near $0$
by interpolating with its differential $d\phi_0$ at $0$.
Indeed, using a rotationally symmetric smooth test function $\theta$ on $D'$,  
we put 
\begin{equation*}
\phi_t :=\phi+t\theta_{\la}(d\phi_0-\phi)
\qquad\qquad 
0\leq t\leq1
\end{equation*}
where $\theta_{\la}(x):=\theta(\frac{x}{\la})$. 
Then 
$\|d\phi_t-d\phi\|\leq C\la t$,
and for sufficiently small $\la>0$ 
we obtain an isotopy.

Assume now that $\phi$ agrees near $0$
with an orientation preserving $\rho$--equivariant linear map $A$. 
If $\rho(H)$ contains reflections,
then $A$ preserves a line 
and, if $\rho(H)$ also contains a rotation of order $\geq3$, 
is a dilation.
One can equivariantly isotope $\phi$ 
to make it equal to $\pm \id$ near $0$.
If $\rho(H)$ preserves orientation,
then $A$ is a homothety if $|\rho(H)|\geq3$,
and an arbitrary orientation preserving linear automorphism otherwise. 
In both cases, 
one can equivariantly isotope $\phi$ 
to make it equal to $\id$ near $0$.
\end{proof}

Due to Lemma~\ref{isotoid_pt}, 
we may assume after applying a suitable isotopy,
that $\phi$ equals the identity in a neighborhood 
of every cone point and every corner reflector. 
By sliding along the singular edges,  
we may isotope $\phi$ 
further so that it fixes the singular locus pointwise.
Since $\phi$ preserves orientation, if ${\mathcal O}$ has boundary reflectors,
an argument as in the proof of Lemma~\ref{isotoid_pt} (but simpler)
allows to isotope $\phi$ 
so that it fixes a neighborhood of the singular locus pointwise. 

\begin{lem}
\label{isotoid_sphdisk}
(i)
Suppose that ${\mathcal O}$ is a 2--sphere with two or three cone points
and that $\phi\co {\mathcal O}\to{\mathcal O}$ 
is an orientation preserving diffeomorphism 
which fixes every cone point.
Then $\phi$ is isotopic to the identity. 

(ii)
Suppose that $D$ is a closed 2--disk with at most one cone point
and that $\phi\co D\to D$ is an orientation preserving diffeomorphism
fixing a neighborhood of $\D D$ pointwise. 
(It must fix the cone point if there is one.) 
Then $\phi$ is isotopic to the identity
by an isotopy supported on the interior of $D$. 

(iii)
Suppose that $M$ is a closed M\"obius band
and that $\phi\co M\to M$ is a diffeomorphism
fixing a neighborhood of $\D M$ pointwise. 
Then $\phi$ is isotopic to the identity
by an isotopy supported on the interior of $M$. 
\end{lem}
\begin{rem}
In case (i) 
and in case (ii) when there is a cone point, 
the isotopy will in general not be supported away from the cone points
but must rotate around them (Dehn twists!). 
\end{rem}
\begin{proof}
(i) It suffices to consider the case of three cone points.
Let us denote them by $p,q$ and $r$.
We choose a smooth arc $\ga$ from $p$ to $q$ avoiding $r$.
We may assume that the image arc $\phi(\ga)$ is transversal to $\ga$. 
Let $x_1,\dots,x_n$, $n\geq0$, denote the interior intersection points 
of $\ga$ and $\phi(\ga)$ 
numbered according to their order along $\phi(\ga)$.
Set $x_{0}=p$ and $x_{n+1}=q$.

Suppose that the number $n$ of transverse intersections  
cannot be decreased by isotoping $\phi$,
and that $n\geq1$.
For $i=0$ and $n$,  
we consider the subarc $\al_i$ of $\phi(\ga)$ from $x_i$ to $x_{i+1}$. 
Let $\beta_i$ be the subarc of $\ga$ with the same endpoints. 
We denote by $D_i$ the disk bounded by the circle $\al_i\cup\beta_i$ 
whose interior is disjoint from $\ga$, 
i.e.\ such that $p\not\in D_n$ and $q\not\in D_0$. 
This choice implies 
that the disks $D_0$ and $D_n$ have disjoint interiors,
because $\D D_0\cap \Int(D_n)=\emptyset$ 
and $\D D_n\cap \Int(D_0)=\emptyset$. 
(Note however, that they may contain other subarcs of $\phi(\ga)$.) 
It follows that at least one of the disks $D_0$ and $D_n$ 
does not contain the third cone point $r$.  
We then can push this disk through $\ga$ 
by applying a suitable isotopy of $\phi$
and thereby reduce the number of intersection points,
a contradiction. 
(Obviously, we can do the same if there are just two cone points $p$ and $q$.) 

This shows that we can isotope $\phi$
such that $\phi(\ga)$ and $\ga$ have no interior intersection points at all.
We may isotope further so that $\phi(\ga)=\ga$ and, 
since $\phi$ preserves orientation, 
even so that $\phi$ fixes a neighborhood of $\ga$ pointwise,
cf.\ Lemma~\ref{isotoid_pt}.
This reduces our assertion to case (ii). 

(ii)
Suppose first that $D$ has a cone point $p$. 
We then proceed as in case (i).
This time we choose $\ga$ to connect a point on $\D D$ to $p$.
If $\phi(\ga)$ intersects $\ga$ transversally,
we can remove all intersection points by applying suitable isotopies
supported away from $\D D$. 
Since $\phi$ preserves orientation, 
we can isotope $\phi$ such that it equals the identity 
in a neighborhood of $\D D\cup\ga$. 
This reduces our assertion 
to the case of a smooth disk. 
In this case,
the result is proven in Munkres \cite[Theorem 1.3]{Munkres},
see also Thurston \cite[end of proof of Theorem 3.10.11]{Thurston3M}.

(iii)
This fact is proven in Epstein \cite[Theorem 3.4]{Epstein_Curves}. 
\end{proof} 

The claim of Proposition~\ref{diffeo2orb} now follows
from part (i) of Lemma~\ref{isotoid_sphdisk}
if ${\mathcal O}$ is a sphere with two or three cone points,
from part (iii) if ${\mathcal O}$ is a projective plane with one cone point,
and from part (ii) if ${\mathcal O}$ has boundary reflectors. 
Note that we have to use isotopies which rotate around cone points. 
Isotopies supported away from the singular locus do not suffice. 

This concludes the proof of Proposition~\ref{diffeo2orb}.
\end{proof}  

We will use the following consequences of Proposition~\ref{equivdiffeo2sph}
in section~\ref{sec:cutoff}.

\begin{cor}
\label{exttoball}
(i) 
Given two isometric actions $\rho_1,\rho_2\co H\acts\bar B^3(1)$ 
of a finite group $H$ on the closed Euclidean unit ball, 
any $(\rho_1,\rho_2)$--equivariant diffeomorphism $\al\co \D B\to\D B$
extends to a $(\rho_1,\rho_2)$--equivariant diffeomorphism 
$\hat\al\co \bar B\to\bar B$. 

(ii) 
The same assertion with $\bar B^3(1)$ replaced by 
the complement $\R P^3-B^3$ of an open round ball with radius $<\frac{\pi}{2}$ 
in projective 3--space equipped with the standard spherical metric. 
\end{cor}
\begin{proof}
In both cases 
we can choose a collar neighborhood $C$ of the boundary sphere
and use Proposition~\ref{equivdiffeo2sph}
to extend $\al$ to a $(\rho_1,\rho_2)$--equivariant diffeomorphism
$C\to C$,
which is isometric on the inner boundary sphere.
It is then trivial to further extend $\al$ 
equivariantly and isometrically to the rest of the manifold. 
\end{proof}

\begin{cor}
\label{acts2s1}
Let $\rho\co H\acts S^2\times S^1$ 
be a smooth action of a finite group 
which preserves the foliation ${\mathcal F}$ 
by the 2--spheres $S^2\times\{\Point \}$. 
Then $\rho$ is standard. 
\end{cor}

We recall that,
as defined in section~\ref{defstan},
an action on $S^2\times S^1$ is standard if and only if 
there exists an invariant Riemannian metric 
locally isometric to $S^2\times\R$.

\begin{proof}
Since $\rho$ preserves ${\mathcal F}$,
it induces an action $\bar\rho\co H\acts S^1$. 
We denote the kernel of $\bar\rho$ by $H_0$. 

There exists a $\bar\rho$--invariant metric on $S^1$. 
We choose a finite $\bar\rho$--invariant subset $A\subset S^1$ as follows. 
If $\bar\rho$ acts by rotations, 
let $A$ be an orbit. 
Otherwise, if $\bar\rho$ acts as a dihedral group, 
let $A$ be the set of all fixed points of reflections in $\bar\rho(H)$. 
Let $g_0$ be a $\rho$--invariant spherical metric on 
the union $\Si\subset S^2\times S^1$ 
of the ${\mathcal F}$--leaves corresponding to the points in $A$. 

There exists a $\rho$--invariant line field ${\mathcal L}$
transversal to the foliation ${\mathcal F}$. 
Following the integral lines of ${\mathcal L}$ 
we obtain $H_0$--equivariant self-diffeomorphisms of $\Si$. 
Using Proposition~\ref{equivdiffeo2sph},
we can modify ${\mathcal L}$ 
so that these self-diffeomorphisms become $g_0$--isometric. 
Then a $\rho$--invariant metric locally isometric to $S^2(1)\times\R$
can be chosen so that the ${\mathcal F}$--leaves are 
totally-geodesic unit spheres 
and ${\mathcal L}$ is the line field orthogonal to ${\mathcal F}$.
\end{proof}

\begin{rem}
We will show later, 
see Theorem~\ref{thm:s2s1}, 
that the same conclusion holds without assuming 
that ${\mathcal F}$ is preserved by the action. 
\end{rem}

\subsection{Equivariant connected sum (decomposition)} 
\label{sec:eqcs}

We fix a finite group $G$ 
and consider smooth actions $\rho\co G\acts M$ 
on closed (not necessarily connected) 3--manifolds. 

Suppose that we are given a $\rho$--invariant 
finite family of pairwise disjoint embedded 2--spheres $S_i^2\subset M$. 
Cutting $M$ along $\cup_iS_i$ yields a compact manifold $\check M$
with boundary.
To every sphere $S_i$ 
correspond two boundary spheres $S_{i1}$ and $S_{i2}$ of $\check M$.
The action $\rho$ induces a smooth action $\check\rho\co G\acts\check M$.
Let $G_i:=\Stab_G(S_{i1})=\Stab_G(S_{i2})$. 
For every boundary sphere $S_{ij}$ 
we choose a copy $\bar B_{ij}$ of the closed unit 3--ball 
and an orthogonal action $\check\rho_{ij}\co G_i\acts B_{ij}$ 
such that there exists a $G_i$--equivariant diffeomorphism 
$\phi_{ij}\co \D B_{ij}\buildrel\cong\over\to S_{ij}$.
We attach the balls $\bar B_{ij}$ to $\check M$
using the $\phi_{ij}$ as gluing maps
and obtain a closed manifold $M'$. 
The action $\check\rho$ extends to a smooth action
$\rho'\co G\acts M'$,
and the smooth conjugacy class of $\rho'$ does not depend on the choice 
of the gluing maps $\phi_{ij}$, 
compare Corollary~\ref{exttoball} (i). 
We call $\rho'$ an 
{\em equivariant connected sum decomposition} of $\rho$. 
(Note that the spheres $S_i$ are allowed to be non-separating.) 

This construction is reversed 
by the {\em equivariant connected sum} operation. 
Suppose that ${\mathcal P}=\{P_i:i\in I\}$ 
is a finite $G$--invariant family of pairwise disjoint two point subsets 
$P_i=\{x_i,y_i\}$, $x_i\neq y_i$, of $M$. 
Then there are induced actions of $G$ on ${\mathcal P}$
and on $\cup_{i\in I}P_i$. 
Let $G_i:=\Stab_G(x_i)=\Stab_G(y_i)$. 
We suppose that for every $i\in I$ 
the actions 
$d\rho_{x_i}\co G_i\acts T_{x_i}M$ 
and 
$d\rho_{y_i}\co G_i\acts T_{y_i}M$ 
are equivalent via a linear isomorphism 
$\al_{x_i}\co T_{x_i}M\to T_{y_i}M$, respectively, 
$\al_{y_i}=\al_{x_i}^{-1}\co T_{y_i}M\to T_{x_i}M$.
More than that, 
we require that the family of the $\al$'s is $G$--equivariant,
i.e.\ if $\{z,w\}$ is one of the pairs $P_i$ and $g\in G$
then $d\rho(g)_w\circ\al_z=\al_{gz}\circ d\rho(g)_z$. 
We denote the collection of the $\al_z$, $z\in{\mathcal P}$, by $\al$. 

The {\em connected sum} of $\rho$ along $({\mathcal P},\al)$ 
is constructed as follows. 
Choose a $G$--invariant auxiliary Riemannian metric on $M$. 
Let $r>0$ be sufficiently small
so that the $2r$--balls around all points $x_i,y_i$
are pairwise disjoint. 
Via the exponential map,
the linear conjugacies $\al_{x_i},\al_{y_i}$ induce a (smooth) conjugacy 
between the actions 
$G_i\acts B_{2r}(x_i)$
and
$G_i\acts B_{2r}(y_i)$. 
We delete the open balls $B_r(x_i)$ and $B_r(y_i)$ 
and glue $G$--equivariantly along the boundary spheres. 
We obtain a new action $\rho_{\mathcal P}\co G\acts M_{\mathcal P}$.
The manifold $M_{\mathcal P}$ admits a natural smooth structure 
such that the action $\rho_{\mathcal P}$ is smooth. 
The smooth conjugacy class of $\rho_{\mathcal P}$ 
depends only on ${\mathcal P}$ and $\al$. 
(We will suppress the dependence on $\al$ in our notation.) 

If $G$ does not act transitively on ${\mathcal P}$,
one can break up the procedure into several steps: 
Suppose that ${\mathcal P}$ decomposes as the disjoint union 
${\mathcal P}={\mathcal P}_1\dot\cup{\mathcal P}_2$
of $G$--invariant subfamilies ${\mathcal P}_i$.
Then $\rho_{\mathcal P}\co G\acts M_{\mathcal P}$
is smoothly conjugate to 
$(\rho_{{\mathcal P}_1})_{{\mathcal P}_2}\co G\acts 
(M_{{\mathcal P}_1})_{{\mathcal P}_2}$,
$\rho_{\mathcal P}\cong(\rho_{{\mathcal P}_1})_{{\mathcal P}_2}$. 

It will be useful to consider the finite graph $\Ga$ 
associated to $M$ and ${\mathcal P}$ as follows:
We take a vertex for each connected component of $M$ 
and draw for every $i$ an edge between the vertices 
corresponding to the components containing $x_i$ and $y_i$. 
(Edges can be loops, of course.) 
There is a natural action $G\acts\Ga$ induced by $\rho$. 

\medskip
In the following situation, the connected sum is trivial.
\begin{lem}[Trivial summand]
\label{trivsumm}
Suppose that $M$ decomposes as the disjoint union $M=M_1\dot\cup M_2$
of $G$--invariant closed manifolds $M_i$,
i.e.\ the $M_i$ are $G$--invariant unions of connected components of $M$. 
Assume more specifically 
that $M_2$ is a union of 3--spheres, $M_2=\dot\cup_{i\in I}S^3_i$,
that $x_i\in M_1$ and $y_i\in S^3_i$,
and that the action $G\acts M_2$
(equivalently, the actions $G_i\acts S^3_i$)
are standard. 

Then $\rho_{\mathcal P}$ is smoothly conjugate to $\rho|_{M_1}$.
\end{lem}
\begin{proof}
Consider a $G$--invariant family of disjoint small balls 
$B_{2r}(x_i)\subset M_1$ and $B_{2r}(y_i)\subset S^3_i$ as above. 
The $G_i$--actions on 
$\bar B_r(x_i)$ and the complement of $B_r(y_i)$ in $S^3_i$ 
are conjugate. 
Thus, in forming the connected sum, we glue back in what we took out.
\end{proof}

We will be especially interested in the situation when $M_{\mathcal P}$
is {\em irreducible}. 

\begin{prop}
\label{decoirred}
Suppose that $M_{\mathcal P}$ is irreducible and connected.
Suppose furthermore that the action $\rho$ is standard on the union 
of all components of $M$ diffeomorphic to $S^3$. 

(i) If $M_{\mathcal P}\cong S^3$, 
then $\rho_{\mathcal P}$ is standard.

(ii) If $M_{\mathcal P}\not\cong S^3$, 
then there exists a unique connected component $M_0$ of $M$ 
diffeomorphic to $M_{\mathcal P}$.
It is preserved by $\rho$ and 
$\rho_{\mathcal P}$ is smoothly conjugate to $\rho|_{M_0}$. 
\end{prop}
\begin{proof}
Under our assumption, 
the graph $\Ga$ is connected. 
Since $M_{\mathcal P}$ is irreducible, 
$\Ga$ cannot contain cycles or loops 
and thus is a tree. 

(ii)
Since the prime decomposition of $M_{\mathcal P}$ is trivial,
a unique such component $M_0$ of $M$ exists
and all other components are diffeomorphic to $S^3$. 
The vertex $v_0$ of $\Ga$ corresponding to $M_0$ is fixed by $G$.

If $\Ga$ is just a point, the assertion is trivial. 
Suppose that $\Ga$ is not a point. 
We choose $M_2\subset M$ as the union of the $S^3$--components 
which correspond to the endpoints of $\Ga$ different from $v_0$.
Let $M_1=M-M_2$. 
We have $M_0\subseteq M_1$. 
Each component of $M_2$ intersects a unique pair $P_i$ 
in exactly one point.
We denote the subfamily of these $P_i$ by ${\mathcal P}_{\mathrm{out}}$
and put ${\mathcal P}_{\mathrm{inn}}={\mathcal P}-{\mathcal P}_{\mathrm{out}}$.
According to Lemma~\ref{trivsumm}, 
$\rho_{{\mathcal P}_{\mathrm{out}}}\cong\rho|_{M_1}$. 
Furthermore,
$\rho_{\mathcal P}\cong (\rho_{{\mathcal P}_{\mathrm{out}}})_{{\mathcal P}_{\mathrm{inn}}}$
and thus 
$\rho_{\mathcal P}\cong (\rho|_{M_1})_{{\mathcal P}_{\mathrm{inn}}}$.
We may replace $\rho$ by $\rho|_{M_1}$ 
and ${\mathcal P}$ by ${{\mathcal P}_{\mathrm{inn}}}$. 
After finitely many such reduction steps, 
we reach the case when $\Ga$ is a point. 

(i) 
If not all vertices of $\Ga$ are endpoints,
i.e.\ if $\Ga$ is not a point or a segment, 
then we can perform a reduction step as in case (ii). 
We can therefore assume that $\Ga$ is a point or a segment. 

If $\Ga$ is a point, there is nothing to show. 
Suppose that $\Ga$ is a segment, i.e.\ ${\mathcal P}=\{P\}$.  
Then $M$ is the disjoint union of two spheres, 
$M=S^3_1\dot\cup S^3_2$, 
and $P=\{z_1,z_2\}$ with $z_i\in S^3_i$. 
Note that the points $z_1,z_2$ may be switched by $\rho$.
Let $M$ be equipped with a spherical metric such that $\rho$ is isometric. 
Then $\al_{z_1}$ and $\al_{z_2}=\al_{z_1}^{-1}$
are the differentials of an involutive $G$--equivariant isometry $\Phi\co M\to M$
switching $z_1$ and $z_2$.
The action $\rho_{\mathcal P}$ can be obtained by restricting $\rho$ 
to the union of the hemispheres $B_{\frac{\pi}{2}}(\hat z_i)$ 
centered at the antipodes $\hat z_i\in S^3_i$ of $z_i$, 
and gluing the boundary spheres along $\Phi$, 
compare the proof of Lemma~\ref{trivsumm}. 
Thus $\rho_{\mathcal P}$ is standard also in this case. 
\end{proof}

\subsection{Balls invariant under isometric actions on the 3--sphere}
\label{sec:invballs}

In this section we will prove the following auxiliary result
which says that an action on a 3--ball is standard 
if it extends to a standard action on the 3--sphere. 

\begin{prop}
\label{subact}
Suppose that $\rho\co  G\acts S^3$ is an isometric action 
(with respect to the standard spherical metric)
and that $\bar B^3\subset S^3$ is a $\rho$--invariant smooth closed ball. 
Then the restricted action $\rho|_{\bar B}$
is standard. 
\end{prop}

We denote $\Si:=\D B$, $B_1:=B$, $B_2:=S^3-\bar B$.

If $\rho(g)$, $g\in G$, has no fixed point on $\Si$ then,
due to Brouwer's Fixed Point Theorem, 
it has at least one fixed point $p_1$ in $B_1$ 
and one fixed point $p_2$ in $B_2$.
No geodesic segment $\si$ connecting them
can be fixed pointwise, because $\si$ would intersect $\Si$
and there would be a fixed point in $\Si$, a contradiction.
Since $\rho(g)$ is a spherical isometry, 
it follows that $p_1$ and $p_2$ are antipodal isolated fixed points.
Thus $\rho(g)$ is the spherical suspension 
of the antipodal involution on $S^2$. 
In particular, 
$\rho(g)$ has order two and reverses orientation. 

Assume now that $\rho(g)$ does have fixed points on $\Si$. 
Near $\Si$, $\Fix(\rho(g))$ is an interval bundle over $\Fix(\rho(g))\cap\Si$.
To see this, note that 
the normal geodesic through any fixed point on $\Si$ belongs to $\Fix(\rho(g))$. 

If $q$ is an isolated fixed point of $\rho(g)$ on $\Si$ 
then $\rho(g)|_{\Si}$ is conjugate to a finite order rotation
and has precisely two isolated fixed points $q$ and $q'$. 
Moreover, 
$\rho(g)$ is a finite order rotation on $S^3$ with the same rotation angle,
and $\Fix(\rho(g))$ is a great circle. 
We have $\Fix(\rho(g))\cap\Si=\Fix(\rho(g)|_{\Si})=\{q,q'\}$. 

If $\rho(g)$ has no isolated fixed point on $\Si$ 
then, due to the classification of finite order isometries on $S^2$,
the only remaining possibility is that 
$\rho(g)|_{\Si}$ is (conjugate to) a reflection at a circle, 
$\rho(g)$ is a reflection at a 2--sphere
and we have $\Fix(\rho(g))\cap\Si=\Fix(\rho(g)|_{\Si})\cong S^1$.

This gives:

\begin{lem}
\label{observation}
Let $h$ be a $\rho$--invariant auxiliary spherical metric on $\Si$.
Then for any $g\in G$, $\rho(g)$ is isometrically conjugate to the
spherical suspension of $\rho(g)|_{(\Si,h)}$. 
\end{lem}

Note that the conjugating isometry might 
a priori depend on $g$. 

We will see next that 
the action $\rho$ 
is determined by its restriction $\rho|_{\Si}$ to $\Si$.
Let us denote by $\tilde\rho$ the spherical suspension of
$\rho|_{(\Si,h)}$.
Hence $\tilde{\rho}\co G\acts S^{3}$ is an isometric action on the unit
sphere and 
we may regard both actions as representations 
$\rho,\tilde\rho\co G\to O(4)$. 

\begin{lem}
\label{suspense}
The representations $\rho$ and $\tilde\rho$ are isomorphic. 
\end{lem}
\begin{proof}
According to Lemma~\ref{observation},
the characters of the two representations are equal.
Therefore their complexifications $\rho_{\C},\tilde\rho_{\C}\co G\to O(4,\C)$ 
are isomorphic,
see e.g.\ Serre \cite[Corollary 2 of chapter 2.3]{Serre},
i.e.\ there exists a $(\rho,\tilde\rho)$--equivariant 
complex linear isomorphism $A\co \C^4\to\C^4$,
$A\circ\rho_{\C}=\tilde\rho_{\C}\circ A$. 

To deduce that already the real representations are isomorphic, 
we consider the composition 
$a:=Re(A|_{\R^4})\co \R^4\to\R^4$,
of $A$ with the $\tilde\rho$--equivariant 
canonical projection $\C^4\to\R^4$.
It is a $(\rho,\tilde\rho)$--equivariant $\R$--linear homomorphism, 
$a\circ\rho=\tilde\rho\circ a$. 
We are done if $a$ is an isomorphism.
We are also done if $a=0$,
because then $i\cdot A\co \R^4\to\R^4$ is an isomorphism.
Otherwise we have a non-trivial decomposition 
$\R^4\cong \ker(a)\oplus \mathrm{im}(a)$ 
of the $\rho(G)$--module $\R^4$ 
as the direct sum of $\ker(a)$ and the submodule $\mathrm{im}(a)$ of 
the $\tilde\rho(G)$--module $\R^4$.
Hence the representations $\rho$ and $\tilde\rho$ 
contain non-trivial isomorphic submodules. 
We split them off and apply the same reasoning
to the complementary submodules. 
After finitely many steps, the assertion follows. 
\end{proof}

The isomorphism between the representations can be chosen orthogonal. 

As a consequence of Lemma~\ref{suspense}, 
the action $\rho$ has (at least a pair of antipodal) fixed points.
Furthermore, for any $\rho$--fixed point $p$
the induced action $G\acts UT_pS^3$ is smoothly conjugate to $\rho|_{\Si}$.
We show next that there exists a pair of antipodal fixed points 
separated by $\Si$. 

\begin{corollary}
\label{fixptinball}
There exists a pair of antipodal $\rho$--fixed points 
$p_1\in B_1$ and $p_2\in B_2$.
\end{corollary}
\begin{proof}
According to Lemma~\ref{suspense}, 
$\rho$ is a suspension and hence 
$\Fix(\rho(G))$ is a great sphere, a great circle 
or a pair of antipodal points. 

If $\Fix(\rho(G))\cong S^2$,  
then $\rho(G)$ has order two 
and is generated by the reflection at $\Fix(\rho(G))$.
Our earlier discussion implies 
that $\Si$ intersects $\Fix(\rho(G))$ transversally
in one circle $\ga$
which divides $\Fix(\rho(G))$ into the disks $D_i:=B_i\cap \Fix(\rho(G))$. 
Let $\iota_{\Fix(\rho(G))}$ be the antipodal involution on $\Fix(\rho(G))$. 
Since it has no fixed point, 
we have that $\iota_{\Fix(\rho(G))}D_1\not\subseteq D_1$.
This implies that the open set $\iota_{\Fix(\rho(G))}D_1$ intersects $D_2$ 
and there exist antipodal points $p_1\in D_1$ and $p_2\in D_2$ as desired. 

If $\Fix(\rho(G))\cong S^1$ 
then there is a rotation $\rho(g)\in \rho(G)$ with $\Fix(\rho(g))=\Fix(\rho(G))$.
It follows that $\Si$ intersects $\Fix(\rho(G))$ transversally in two points. 
As before,
there exist points $p_1,p_2\in \Fix(\rho(G))$ as desired. 

We can now assume that $\Fix(\rho(G))=\{p,\hat p\}$ is a pair of antipodal points 
and we must show that $\Si$ separates them. 
(Note that $\rho(G)$ cannot fix a point on $\Si$ 
because otherwise $\dim(\Fix(\rho(G)))\geq1$. 
Thus $p,\hat p\not\in\Si$.)

If $\rho(G)$ has order two and is generated by the involution 
with isolated fixed points $p$ and $\hat p$,
then Brouwer's Fixed Point Theorem implies 
that each ball $B_i$ contains one of the fixed points,
and we are done in this case. 

Otherwise $\rho(G)$ contains nontrivial orientation preserving isometries. 
Any such element $\rho(g)$ 
is a rotation whose axis $\Fix(\rho(g))$ is a great circle through $p$ and $\hat p$. 
If there exists $\rho(g')\in \rho(G)$ preserving $\Fix(\rho(g))$
and such that $\rho(g')|_{\Fix(\rho(g))}$ is a reflection at $\{p,\hat p\}$
then we are done
because the ($\rho(g')$--invariant) 
pair of points $\Si\cap \Fix(\rho(g))$ separates $p$ and $\hat p$.
Let us call this situation (S).

We finish the proof by showing that (S) always occurs.
Consider the induced action $d\rho_{p}\co  G\acts UT_pS^3$ 
on the unit tangent sphere in $p$
and in particular on the nonempty finite subset $F$ 
of fixed points of nontrivial rotations in $d\rho_{p}(G)$.
We are in situation (S)  
if and only if some $d\rho_{p}(G)$--orbit in $F$ contains a pair of antipodes. 
Suppose that we are not in situation (S).
Then $F$ must decompose into an even number of $\rho(G)$--orbits,
in fact,
into an even number of $H$--orbits for any subgroup $H\leq d\rho_{p}(G)$. 
(The action $d\rho_{p}$ commutes with the antipodal involution of $UT_{p}S^{3}$.)
Let $G_+\leq d\rho_{p}(G)$ be the subgroup of orientation preserving isometries. 
It follows that the spherical quotient 2--orbifold $UT_pS^3/G_+$
has an even non-zero number of cone points 
and hence is a sphere with two cone points,
i.e.\ the spherical suspension 
of a circle of length $\frac{2\pi}{m}$, $m\geq2$.
So, $F$ is a pair of antipodes.
$d\rho_{p}(G)$ cannot interchange them 
because we are not in situation (S).
On the other hand,
$d\rho_{p}(G)$ cannot fix any point on $UT_pS^3$ since the fixed point
set of $\rho(G)$ on $S^3$ is 0-dimensional.
We obtain a contradiction and conclude that we are always in situation (S).
\end{proof}

In view of Lemma~\ref{suspense} and Corollary~\ref{fixptinball}
we reformulate Proposition~\ref{subact} as follows.
By removing small invariant balls around $p_1,p_2$ 
and replacing the metric,
we convert $\rho$ into an isometric action 
$\rho_1\co G\acts S^2\times I$ 
on the product of the unit 2--sphere with $I=[0,1]$
which acts trivially on $I$ 
(i.e.\ preserves top and bottom).
We denote by $\bar\rho_1$ 
the projection of the $G$--action $\rho_1$ to $S^2$.
We regard $\Si$ as a 
$\rho_1$--invariant embedded 2--sphere $\Sigma\subset S^2\times I$. 
Since the actions 
${\rho_1}|_{\Si}$ and $\bar\rho_1$ are conjugate,
we have a $(\rho_1,\bar\rho_1)$--equivariant diffeomorphism 
$\psi\co \Si\to S^2$. 

Proposition~\ref{subact} follows from: 

\begin{lem}
\label{isotopic}
$\Si$ is $\rho_1$--equivariantly isotopic to a horizontal sphere $S^2\times t$.
\end{lem}
\begin{proof}
We know from Corollary~\ref{fixptinball}:
For any rotation $\rho(g)$ on $S^3$ 
each of the two intervals 
$p\times I$ and $\hat p\times I$ fixed by $\rho_1(g)$
intersects $\Si$ transversally in one point. 
For any reflection $\rho(g)$ at a 2--sphere in $S^3$,
$\bar\rho_1(g)$ acts on $S^2$ as the reflection at a great circle $\mu$,
and $\Si$ intersects $\mu\times I$ transversally in one circle. 

{\em Case 1.}
Suppose first that $\rho(G)$ does contain reflections at 2--spheres. 
The mirror great circles in $S^2$ 
of the corresponding reflections in $\bar\rho_1(G)$ 
partition $S^2$ into isometric convex polygonal tiles
which are either hemispheres, bigons or triangles.
(Polygons with more than three vertices are ruled out by Gau\3-Bonnet.) 
Every intersection point of mirror circles 
is a fixed point of a rotation $\bar\rho_1(g)$.
Further fixed points of rotations $\bar\rho_1(g)$ 
can lie in the (in)centers of the tiles. 
This can occur only if the tiles are hemispheres, bigons or 
equilateral right-angled triangles. 
(No midpoint of an edge can be fixed by a rotation $\bar\rho_1(g)$
because then another mirror circle would have to run through this fixed point,
contradicting the fact that it is not a vertex.)

Denote by $\Ga\subset S^2$ the union of the mirror circles $\mu$ 
of all reflections in $\bar\rho_1(G)$. 
It is a $\bar{\rho}_{1}$--invariant great circle or connected geodesic graph. 
Note that, in the second case,
when $\si$ is an edge of $\Ga$ contained in the mirror $\mu$
then the circle $\Si\cap\mu\times I$ 
intersects both components of $\D\si\times I$ 
transversally in one point and 
$\Si\cap\si\times I$ is a curve connecting them. 
It follows, in both cases,
that $\Si$ can be $G$--equivariantly isotoped
to be horizontal over a neighborhood of $\Ga$. 
The tiles (components of $S^2-\Ga$) 
are topologically disks,
and for any tile $\tau$ the intersection 
$\Si\cap\bar\tau\times I$ is a disk 
(since it is bounded by a circle).

To do the isotopy over the 2--skeleton, let us divide out the $G$--action. 
We consider the spherical 2--orbifold ${\mathcal O}^2=S^2/\bar\rho_1(G)$.
It has reflector boundary, 
its underlying topological space is the 2--disk,
and possibly there is one cone point in the interior. 
The quotient 2--orbifold $\Si/\rho_{1}(G)\subset{\mathcal O}\times I$
is diffeomorphic to ${\mathcal O}$ 
and the embedding is horizontal over $\D{\mathcal O}$.
If ${\mathcal O}$ has no cone point in the interior
then it follows with Alexander's Theorem,
see e.g.\ Hatcher \cite[Theorem 1.1]{Hatcher}, 
that $\Si/\rho_1(G)$ can be made horizontal 
by an isotopy fixing the boundary, 
so the assertion of the Proposition holds in this case. 
If ${\mathcal O}$ has one cone point in its interior,
the assertion is a consequence of the annulus case 
of the following standard result:

\begin{sublem}
\label{horsurf}
Let $\Si_1\cong S^2-\cup_{i=0}^nD_i$, $n\geq0$,  
where the $D_i$ are open disks with disjoint closures. 
Suppose that $\Si_2\subset\Si_1\times [0,1]$ 
is a properly embedded surface,
$\D\Si_2\subset\D\Si_1\times(0,1)$, 
and that $\phi\co \Si_2\to\Si_1$ is a diffeomorphism 
which near the boundary coincides with the canonical projection onto $\Si_1$.

Then $\Si_2$ can be isotoped to be horizontal.
\end{sublem}
\begin{proof}
Let $\al_1,\dots,\al_n$ be disjoint properly embedded arcs in $\Si_1$ 
such that $\D\Si_1\cup\al_1\cup\dots\cup\al_n$ is connected. 
Then cutting $\Si_1$ along the $\al_i$ yields a disk. 
We may assume that $\Si_2$ intersects the strips $\al_i\times I$ 
transversally. 
Each intersection $\Si_2\cap\al_i\times I$ consists of an arc $\beta_i$ 
connecting the components of $\D\al_i\times I$ 
and finitely many, possibly zero, circles.
Note that $\D\Si_2\cup\beta_1\cup\dots\cup\beta_n$ 
is connected and hence 
cutting $\Si_2$ along the $\beta_i$ also yields a disk. 

Suppose that $\ga\subset\Si_2\cap\al_i\times I$ is a circle.
It lies in the complement of $\cup_j\beta_j$ 
and thus bounds a disk $D\subset\Si_2$. 
Suppose in addition 
that $\ga$ is innermost on $\Si_2$ 
in the sense that 
$D\cap(\cup_j\al_j\times I)$ is empty. 
Let $D'$ be the disk bounded by $\ga$ in $\al_i\times I$.  
($\ga$ is not necessarily innermost in $\al_i\times I$, too,
i.e.\ $D'$ may intersect $\Si_2$ in other circles.) 
It follows from Alexander's Theorem 
(applied to the ball obtained from cutting $\Si_1\times I$ 
along the strips $\al_i\times I$)
that the embedded 2--sphere 
$D\cup D'$ bounds a 3--ball, 
and by a suitable isotopy we can reduce 
the number of circle components of the intersection of $\Si_2$ 
with $\cup_j\al_j\times I$. 
After finitely many steps
we can achieve that $\Si_2\cap\al_i\times I=\beta_i$ for all $i$. 
After a suitable isotopy of $\Si_2$ rel $\D\Si_2$ 
we may assume that the projection onto $\Si_1$ 
restricts to diffeomorphisms $\beta_i\buildrel\cong\over\to\al_i$. 

We now cut $\Si_1\times I$ along the strips $\al_i\times I$ 
and obtain a ball $\check\Si_1\times I$.
The surface $\Si_2$ becomes a properly embedded disk 
$\check\Si_2\subset\check\Si_1\times I$.
Moreover, 
$\D\check\Si_2\subset\D\check\Si_1\times (0,1)$ 
and the projection onto $\check\Si_1$ induces a diffeomorphism 
$\D\check\Si_2\to\D\check\Si_1$. 
Applying Alexander's Theorem once more,
we conclude that there exists an isotopy of $\check\Si_2$ rel $\D\check\Si_2$
which makes $\check\Si_2$ transversal to the interval fibration,
i.e.\ such that the projection onto $\check\Si_1$ induces a diffeomorphism 
$\check\Si_2\to\check\Si_1$. 
The assertion follows. 
\end{proof}

We continue the proof of Lemma~\ref{isotopic}. 

{\em Case 2.}
If $\rho(G)$ is the group of order two generated by an involution of $S^3$ 
with two fixed points 
then the assertion follows from Livesay \cite[Lemma 2]{Livesay}. 
(See also Hirsch and Smale \cite{HirschSmale}.) 
This finishes the proof in the case 
when $\rho$ does not preserve orientation.

{\em Case 3.}
We are left with the case when $\rho$ preserves orientation.
We consider the nonempty finite set $F\subset S^2$
of fixed points of nontrivial rotations in $\bar\rho_1(G)$.
We recall that $\Si$ intersects every component of $F\times I$ 
transversally in one point. 
Let $\dot S\subset S^2$ be the compact subsurface obtained from removing 
a small (tubular) neighborhood around $F$.
Let $\dot\Si=\Si\cap(\dot S\times I)$. 
As above, we divide out the $G$--action 
and consider the properly embedded surface
$\dot\Si/\rho_1(G)\subset\dot S\times I/\rho_1(G)$.
Its boundary is contained in 
$\D\dot S\times(0,1)/\rho_1(G)$.
Note that $\dot\Si/\rho_1(G)\cong\dot S/\bar\rho_1(G)$ 
because the actions $\rho_1|_{\Si}$ and $\bar\rho_1$ are conjugate. 
These surfaces are spheres with 2 or 3 disks removed,
corresponding to the fact that $\Si/\rho(G)$ is an oriented spherical 
2--orbifold with cone points, and the number of cone points can only be 2 or 3. 
We can choose orientations on $\Si$ and $S^2$ 
such that the canonical projection 
induces an orientation preserving diffeomorphism 
$\D\dot\Si/\rho_1(G)\to\D\dot S/\bar\rho_1(G)$.
It extends to an orientation preserving diffeomorphism 
$\phi\co \dot\Si/\rho_1(G)\to\dot S/\bar\rho_1(G)$.
Now Sublemma~\ref{horsurf} implies that $\dot\Si/\rho_1(G)$
can be isotoped to be horizontal. 
The assertion follows also in this case.

This concludes the proof of Lemma~\ref{isotopic}.
\end{proof}

\section{Tube--cap decomposition}
\label{sec:tcdec}

In this section, 
we recall some well-known material on Ricci flows
and adapt it to our setting. 

In a Ricci flow with surgery,
the regions with sufficiently large positive scalar curvature 
are well approximated, up to scaling, 
by so-called $\kappa$-- and standard solutions.
These local models are certain special Ricci flow solutions 
with time slices of nonnegative sectional curvature.
Some of their properties are summarized in section~\ref{sec:locmod}. 
Crucial for controlling the singularities of Ricci flow 
and hence also for our purposes, 
is their {\em neck--cap geometry}:  
Time slices of $\kappa$-- or standard solutions are 
mostly necklike, 
i.e.\ almost everywhere almost round cylindrical 
with the exception of at most two regions, so-called ``caps'', of bounded size
(relative to the curvature scale). 

The neck--cap alternative carries over to regions in a Riemannian 3--manifold
which are well approximated by the local models.
One infers that, globally, 
the region of sufficiently large positive scalar curvature 
in a time slice of a Ricci flow consists of tubes and caps, 
see section~\ref{sec:eqtc}. 
The tubes are formed by possibly very long chains of overlapping necks, 
see section~\ref{sec:fol}. 
One has very precise control of their geometry,
namely they are almost cylindrical of varying width. 
In the time slice of a Ricci flow,  
the quality of approximation improves as scalar curvature increases.
Hence the thinner a tube becomes, 
the better it is approximated by a round cylinder. 
The caps, on the other hand, 
enclose the small ``islands'' far apart from each other 
whose geometry is only roughly known, 
compare section~\ref{sec:cap}. 
Tubes and caps can be adjusted 
to yield an equivariant decomposition, cf.\ section~\ref{sec:eqtc}.

\subsection{Some definitions and notation} 
\label{sec:defnot}

We call a diffeomorphism $\phi\co (M_1,g_1)\to(M_2,g_2)$ 
of Riemannian manifolds an {\em $\eps$--isometry}, $\eps>0$,
if $\phi^*g_2$ is $\eps$--close, in the sense of a strict inequality, to $g_1$ 
in the ${\mathcal C}^{[\frac{1}{\eps}]+1}$--topology. 
We call $\phi$ an {\em $\eps$--homothety}, 
if it becomes an $\eps$--isometry
after suitably rescaling $g_2$.

We say that an action $\rho\co G\acts (M,g)$ 
is {\em $\eps$--isometric}, 
if $\rho(\ga)$ is an $\eps$--isometry for all $\ga\in G$. 

Given a point $x$ with scalar curvature $\scal(x)>0$ in a Riemannian manifold, 
we define the distance from $x$
{\em relative to its curvature scale} 
by $\tilde d(x,\cdot):=\scal^{\half}(x)\cdot d(x,\cdot)$.
Accordingly, we define the relative {\em radius} of a subset $A$ by 
$\rrad(x,A):=\sup\{\tilde d(x,y):y\in A\}$, 
and the {\em ball} $\tildeB(x,r):=\{\tilde d(x,\cdot)<r\}$. 

We say that the pointed Riemannian manifold $(M_1,x_1,g_1)$
{\em $\eps$--approximates} $(M_2,x_2,g_2)$
if $\scal(x_1)>0$ and 
if there exists an $\eps$--homothety
$\phi\co (\tildeB_{1/\eps}^{M_1}(x_1),\scal(x_1)g_1)\to 
(V_2,g_2)$ 
onto an open subset $V_2$ of $M_2$
with $\phi(x_1)=x_2$. 
We will briefly say 
that $(M_1,x_1,g_1)$ is {\em $\eps$--close} to $(M_2,x_2,g_2)$.
Note that this definition is scale invariant, whereas the definition
of $\eps$--isometry is not.
\begin{defn}[neck]
\label{def:neck}
Let $(M^3,g)$ be a Riemannian 3--manifold.
We call an open subset ${\mathcal N}\subset M$
an {\em $\eps$--neck}, $\eps>0$,  
if there exists an $\eps$--homothety
\begin{equation}
\label{eq:neck} 
\phi\co S^2(\sqrt{2})\times(-\tfrac{1}{\eps},\tfrac{1}{\eps})\to {\mathcal N}
\end{equation}
from the standard round cylinder 
of scalar curvature 1 and length $\frac{2}{\eps}$ onto ${\mathcal N}$.
We refer to $\phi$ as a {\em neck chart}
and to a point $x\in\phi(S^2(\sqrt{2})\times\{0\})$ 
as a {\em center} of the $\eps$--neck ${\mathcal N}$.
\end{defn}
Note that throughout this paper, all approximations are applied to
time-slices only. 
In particular, 
the necks considered here are not strong necks in the sense of 
Perelman \cite{Perelman_surgery}. 

We denote the open subset consisting of all centers of $\eps$--necks
by $M_{\eps}^{\neck}$, 
and its complement by $M_{\eps}^{\nn}$. 
We measure the {\em necklikeness} in a point $x$ by 
$\nu(x):=\inf\{\eps>0:x\in M_{\eps}^{\neck}\}$. 
For a neck chart (\ref{eq:neck}) one observes that 
$\nu<\frac{1}{1-|a|}\eps$ 
on $\phi(S^2(\sqrt{2})\times\{\frac{a}{\eps}\})$ 
for $-1<a<1$. 
Thus $\nu\co M\to[0,\infty)$ is a continuous function. 
We have that $M_{\eps}^{\neck}=\{\nu<\eps\}$. 
If $\nu$ attains the value zero in some point $x$
then that connected component of $M$ is homothetic to the 
complete round standard cylinder $S^2(\sqrt{2})\times\R$
and $\nu\equiv0$ there. 

\medskip
The following notion will be used to describe 
the non-necklike regions near singularities of Ricci flows, 
compare Propositions~\ref{capsinmodel} and~\ref{caps} below.

\begin{defn}[$(\eps,d)$--cap]
\label{def:cap}
An incomplete Riemannian 3--manifold $C$ 
with strictly positive scalar curvature,
which is diffeomorphic to $B^3$ or $\R P^3-\bar B^3$, 
is called an 
{\em $(\eps,d)$--cap} centered at the point $x$ 
if the following holds:
There exists an $\eps$--neck ${\mathcal N}\subset C$ 
centered at a point $z$ with $\tilde d(x,z)=d$ 
which represents the end of $C$. 
Furthermore, 
$x\not\in{\mathcal N}$ and the compact set $C-{\mathcal N}$ 
is contained in $\tildeB(x,d)$. 
\end{defn}

Note that unlike other authors 
we prescribe a fixed relative diameter for caps
instead of just an upper diameter bound. 
However, this difference is inessential 
because in the local models of sufficiently large diameter 
the non-necklike region consists of at most two components 
of relative bounded diameter, cf.\ Proposition~\ref{capsinmodel}. 
Thus the diameter of caps may be adapted by 
extending their necklike ends.

\subsection{Properties of $\kappa$-- and standard solutions}
\label{sec:locmod}

$\kappa$-- and standard solutions serve as the local models 
for the regions of large positive scalar curvature 
in Ricci flows with surgery.  
Detailed information can be found in 
Kleiner and Lott \cite[Sections 38--51 and 59--65]{KL}, Morgan and Tian
\cite[Chapter 9 and 12]{MorganTian_poinc} and Bamler \cite[Sections 5 and 7.3]{Bamler}. 
We summarize some of their properties most relevant to us.
All $\kappa$--solutions considered below will be 3--dimensional, orientable
and connected. 
The standard solutions are assumed to have a fixed initial metric.

{\em Rigidity.}
The time slices of $\kappa$-- and standard solutions 
have nonnegative sectional curvature.
The time-$t$ slices, $t>0$, 
of standard solutions have strictly positive sectional curvature.
If the sectional curvature of a time slice of a $\kappa$--solution
is not strictly positive,
then the $\kappa$--solution is a shrinking round cylinder 
or its orientable smooth $\Z_2$--quotient. 
In particular, its time slices are noncompact.
(Note that a cyclic quotient of the shrinking round cylinder 
is {\em not} a $\kappa$--solution because time slices far back in the past
are arbitrarily collapsed.) 

{\em Topological classification.}
The topology of the time slice of a $\kappa$--solution 
with strictly positive sectional curvature can be derived from general results 
about positively curved manifolds. 
It is diffeomorphic to $\R^3$ in the noncompact case, see Cheeger--Gromoll 
\cite{soul} and Gromoll--Meyer \cite{poscurv},
and to a spherical space form $S^3/\Ga$ in the compact case by Hamilton
\cite{Hamilton_posric}. 
The time slices of standard solutions are $\cong\R^3$ by definition. 

{\em Universal noncollapsedness.}
There exists a universal constant $\kappa_0>0$ such that 
any $\kappa$--solution is a $\kappa_0$--solution 
unless it is a shrinking spherical space form 
(with large fundamental group). 
Standard solutions are uniformly noncollapsed, as well. 

{\em Compactness.}
The space of pointed $\kappa_0$--solutions 
equipped with the ${\mathcal C}^{\infty}$--topology 
is compact modulo scaling. 
Also the space of pointed standard solutions 
(with fixed initial condition and 
with the tip of the time-$0$ slice as base point) 
is compact \cite[Lemma 64.1]{KL}. 

Let $(N_i,x'_i)$ be a sequence 
of pointed time-$t_i$ slices of standard solutions
and suppose that $\lim_{i\to\infty}t_i=t_{\infty}\in[0,1]$. 
Then, after passing to a subsequence, 
the renormalized time slices 
$\scal(x'_i)^{\half}\cdot(N_i,x'_i)$ 
converge to a time slice $(N_{\infty},x'_{\infty})$
of a $\kappa_0$-- or a renormalized standard solution. 
The limit is a time slice of a $\kappa_0$--solution 
if $t_{\infty}=1$ or if $x'_i\to\infty$ 
(on the manifold underlying the initial condition of standard solutions),
cf.\ \cite[Lemmas 61.1 and 63.1]{KL}.
In the latter case, the limit is the round cylinder with scalar curvature 
$\equiv1$. 
In particular, 
the space of all (curvature) renormalized pointed time slices 
of $\kappa_0$-- or standard solutions is compact. 

These compactness results yield uniform bounds for 
all (scale invariant) geometric quantities,
compare e.g.\ Addendum~\ref{add:bounds} below.

{\em Mostly necklike.}
Time slices of $\kappa$-- or standard solutions are
almost everywhere almost round cylindrical 
with the exception of at most two caps of bounded size. 
More precisely, one has the following information: 
\begin{prop}[Caps in local models]
\label{capsinmodel}
For any sufficiently small $\eps>0$ 
there exist constants $D'(\eps)>d'(\eps)>0$ such that the following hold: 

(i) Suppose that 
$(N,x')$ is a pointed time slice of a $\kappa_0$-- or standard solution 
and that $x'\in N_{\eps}^{\nn}$ with $\rrad(x',N)>D'$. 
Then $x'$ is the center of an $(\eps,d')$--cap $C\subset N$.
Moreover, 
if ${\mathcal N}\subset C$ is an $\eps$--neck 
representing the end of $C$
as in definition~\ref{def:cap},
then ${\mathcal N}\subset N_{\eps}^{\neck}$. 

(ii) If $\hat C\subset N$ is another $(\eps,d')$--cap 
centered at a point $\hat x'\not\in C$, 
then $\hat C\cap C=\emptyset$. 
In this case, $N$ is the time slice of a compact $\kappa_0$--solution
and $N_{\eps}^{\nn}\subset C\cup\hat C$.
\end{prop}
Regarding the geometry of the caps, 
the compactness theorems for the local models 
imply the existence of curvature and diameter bounds.
(This enters already in part (ii) of the previous Proposition.)

\begin{add}[Uniform geometry of caps]
\label{add:bounds}
There exist constants 
$c'_1(\eps)$, $c'_2(\eps)$, $c'_3(\eps,D'')>0$ and $\bar d'(\eps)>0$ 
such that the following holds: 

On the $(\eps,d')$--cap $C$ centered at $x'$ we have 
$c'_1\scal(x')\leq \scal\leq c'_2\scal(x')$. 
Moreover $\rrad(x',C)<\bar d'$. 

If $N$ is compact with $\rrad(x',N)<D''$,  
then $sec\geq c'_3\scal(x')$ on $N$. 
\end{add}

These facts follow from \cite[Theorem 9.93]{MorganTian_poinc}, 
\cite[Corollary 48.1, Lemma 59.7]{KL}, or 
\cite[Lemma 5.4.10, Theorems 5.4.11 and 5.4.12]{Bamler} for
$\kappa$--solutions and from 
\cite[Lemma 63.1]{KL} or
\cite[Theorem 7.3.4]{Bamler} for standard solutions.

\subsection{Foliating the necklike region} 
\label{sec:fol}

Let $(M^3,g)$ be a Riemannian 3--manifold. 
In this section we discuss the global geometry of the necklike region 
and explain that chains of $\eps$--necks fit to almost cylindrical tubes,
possibly long and of varying width.  

In the following,
$\eps_0\in(0,\frac{1}{2008}]$ 
will be a universal sufficiently small positive constant.  

We call a unit tangent vector $v$ at a point $x\in M_{\eps_0}^{\neck}$ 
a {\em distant direction}
if there exists a geodesic segment of length 
$>\scal^{-\half}(x)\frac{1}{2\nu(x)}$ 
starting from $x$ in the direction $v$.
The smaller $\nu(x)$, 
the closer any two distant directions $v_1$ and $v_2$ in $x$ 
are up to sign,
$v_1\simeq\pm v_2$. 
In any point $x\in M_{\eps_0}^{\neck}$ 
exists a pair of almost antipodal distant directions. 

Let $x\in M_{\eps}^{\neck}$, $0<\eps\leq\eps_0$,  
and let $\phi$ be an associated neck chart, cf.\ (\ref{eq:neck}). 
Consider the composition 
$h:=\pi_{(-\frac{1}{\eps},\frac{1}{\eps})}\circ\phi^{-1}$ 
where $\pi_{(-\frac{1}{\eps},\frac{1}{\eps})}\co 
S^2(\sqrt{2})\times(-\frac{1}{\eps},\frac{1}{\eps})
\to(-\frac{1}{\eps},\frac{1}{\eps})$ 
denotes the projection onto the interval factor.
If $\eps$ is sufficiently small,
then the level sets $h^{-1}(t)$ 
are almost totally geodesic and almost round 2--spheres 
with scalar curvature $\simeq \scal(x)$. 
Those not too far from $x$,
say with $-\frac{1}{2\eps}<t<\frac{1}{2\eps}$, 
are almost orthogonal to distant directions in $x$. 
Note that $\|dh\|\simeq \scal^{\half}$. 

Let $h_i$ ($i=1,2$)
be two such functions associated to $\eps$--neck charts $\phi_i$, 
and let 
$V_i=\phi_i(S^2(\sqrt{2})\times(-\frac{1}{2\eps},\frac{1}{2\eps}))$
be the central halves of the corresponding necks ${\mathcal N}_i$. 
Note that the coordinate change 
$\phi_2^{-1}\circ\phi_1$ 
between the two neck charts
is close to an isometry of the standard round cylinder 
(in a ${\mathcal C}^k$--topology for large $k$). 
In particular, 
after adjusting the signs of the $h_i$ if necessary,
the 1--forms $dh_i$ are close to each other 
on the overlap $V_1\cap V_2$, 
and so are the plane fields $\ker\,dh_i$ 
and the $h_i$--level spheres through any point $x\in V_1\cap V_2$. 
The latter ones can be identified e.g.\ by following the gradient flow lines 
of $h_1$ or $h_2$. 

To obtain the global picture,
we now cover $M_{\eps_0}^{\neck}$ with $\eps$--necks.
By interpolating the associated (pairs of) 1--forms $\pm dh$ 
(which may be understood as sections of $T^*M/\{\pm1\}$) 
we will obtain a global foliation  
by almost totally geodesic and almost round 2--spheres 
which are cross sections of $\eps$--necks. 
The geometry of the foliation approaches 
the standard foliation of the round
cylinder by totally geodesic 2--spheres
as the necklikeness approaches zero.

\begin{lem}[Foliation by 2--spheres]
\label{folneck}
There exist an open subset $F$ with 
$\ol{M_{\eps_0}^{\neck}}\subseteq F\subseteq M$,
a closed smooth 1--form $\al$ on $F$ 
and a monotonically increasing function $\theta\co [0,\eps_0)\to[0,\infty)$,
continuous in 0 with $\theta(0)=0$,
such that the following properties are satisfied: 
\BI
\item $\|\al\|\simeq \scal^{\half}$, i.e.\ 
$1-\theta\circ\nu\leq\scal^{-\half}\|\al\|\leq1+\theta\circ\nu$. 
\item The complete integral manifolds of the 
plane field $\ker\,\al$ are 2--spheres foliating $F$.
We denote this foliation by ${\mathcal F}$.
\item Let $x\in F$.
Then, up to scaling,  
the foliation ${\mathcal F}$ is 
on $\tildeB(x,\frac{1}{2\nu(x)})\cap F$ 
$\theta(\nu(x))$--close 
in the 
${\mathcal C}^{[\frac{1}{\theta(\nu(x))}]+1}$--topology 
to the foliation of the standard round cylinder 
by totally-geodesic cross-sectional 2--spheres. 
\EI
\end{lem}
\begin{proof}
We exclude the trivial situation when $\nu$ attains the value zero
and assume that $\nu>0$ everywhere on $M_{\eps_0}^{\neck}$. 

To (almost) optimize the quality of approximation,
we choose for each point $x\in M_{\eps_0}^{\neck}$
a constant $\eps(x)\in(0,\eps_0)$
with $\eps(x)<\frac{101}{100}\nu(x)$
and realize $x$ as the center of an $\eps(x)$--neck ${\mathcal N}_x$
with neck chart $\phi_x$, cf.\ (\ref{eq:neck}). 
For all these necks 
we consider their central halves
$V_x=\phi_x(S^2(\sqrt{2})\times(-\frac{1}{2\eps(x)},\frac{1}{2\eps(x)}))$ 
and thirds 
$W_x=\phi_x(S^2(\sqrt{2})\times(-\frac{1}{3\eps(x)},\frac{1}{3\eps(x)}))$ 
and the exact 1--forms $\al_x=dh_x$ on $V_x$. 
Using a partition of unity on $M$ 
subordinate to the open covering by the $V_x$ 
and the interior of $M-\cup_xW_x$, 
we interpolate the forms $\pm\al_x$ 
to obtain a closed 1--form $\pm\al$ on $\cup_xW_x$ 
(locally) well-defined up to sign,
that is, a section of $T^*M/\{\pm1\}$.
(More precisely, for every point $y\in\cup_xW_x$ 
we choose signs $\eps_{x,y}\in\{\pm1\}$
such that $\eps_{x,y}\al_x\simeq\eps_{x',y}\al_{x'}$
near $y$ if $y\in V_x\cap V_{x'}$,
and then interpolate the forms $\eps_{x,y}\al_x$ near $y$.)

The plane field $\ker(\pm\al)$ on $\cup_xW_x$ is integrable 
and hence tangent to a 2--dimensional foliation ${\mathcal F}$. 
For any $y\in W_x$ the leaf ${\mathcal F}_y$ through $y$ 
is a 2--sphere close to the 2--sphere $h_x^{-1}(t)$ through $y$ 
(because it is a level set of a local primitive $f$ of $\al$, 
and $f$ is close to $\pm h_x+const$).
The approximation and the geometric properties of the leaves improve
as $\nu$ decreases. 
We take $F$ to be a saturated open subset of $\cup_xW_x$
which contains all leaves of ${\mathcal F}$ 
meeting the closure of $M_{\eps_0}^{\neck}$. 
We can also arrange that $\D F$ is a disjoint union of embedded 2--spheres. 

The leaf space of ${\mathcal F}$ is a 1--manifold. 
Therefore we can globally choose a sign for $\pm\al$,
i.e.\ lift the section $\pm\al$ of $T^*M/\{\pm1\}$
to a section $\al$ of $T^*M$.
\end{proof}

Suppose now that in addition we are given 
an isometric action $\rho\co G\acts M$ of a finite group. 
Then the above construction can be done equivariantly.

\begin{lem}[Equivariant foliation]
\label{eqfolneck} 
The set $F$ and its foliation ${\mathcal F}$
obtained in Lemma~\ref{folneck} can be chosen
$\rho$--invariant. 
\end{lem}
\begin{proof} 
The family of embeddings $\phi$ and the partition of unity can
be chosen $G$--invariantly.  Then the resulting section $\pm\al$ of
$T^*M/\{\pm1\}$ is also $G$--invariant.  
\end{proof}

Given a subgroup $H\leq G$, 
we say that an $\eps$--neck ${\mathcal N}\subset M$, 
$0<\eps\leq\eps_0$, 
is {\em $H$--equivariant} or an {\em $(H,\eps)$--neck}, 
if it is $\rho(H)$--invariant as a subset 
and if the neck chart (\ref{eq:neck}) 
can be chosen such that 
the pulled-back action $\phi^*(\rho|_H)$ is {\em isometric}
(with respect to the cylinder metric). 

For $0<\eps\leq\eps_0$, 
we denote by 
$M_{H,\eps}^{\neck}$ 
the subset of centers of $(H,\eps)$--necks. 
It is contained in the union ${\mathcal F}_H$ 
of $H$--invariant leaves of the equivariant foliation ${\mathcal F}$
given by Lemma~\ref{eqfolneck}, 
$M_{H,\eps}^{\neck}\subseteq {\mathcal F}_H\cap M_{\eps}^{\neck}$;  
it is a union of ${\mathcal F}$--leaves 
and open in ${\mathcal F}_H$.
We define the {\em equivariant necklikeness} 
$\nu_H\co M_{H,\eps_0}^{\neck}\to[0,\eps_0)$ analogously
by $\nu_H(x):=\inf\{\eps>0:x\in M_{H,\eps}^{\neck}\}$. 

The next observation says 
that we can replace necks by equivariant ones. 
This becomes relevant 
when one wants to perform surgery on the Ricci flow equivariantly,
cf.\ section~\ref{sec:ex}. 

\begin{lem}[Equivariant necks]
\label{equivneck}
There exists a constant $\eps_0^G\in(0,\eps_0]$ 
and a monotonically increasing function $f\co [0,\eps_0^G]\to[0,\eps_0]$, 
continuous in $0$ with $f(0)=0$, 
such that for any subgroup $H\leq G$ holds 
$\nu_H\leq f\circ\nu$ on 
$M_{\eps_0^G}^{\neck}\cap{\mathcal F}_H\subseteq M_{H,\eps_0}^{\neck}$. 
\end{lem}
\begin{proof}
Let $x\in M_{\eps_0}^{\neck}\cap{\mathcal F}_H$ 
and let us normalize so that $\scal(x)=1$. 
As before, we denote by ${\mathcal F}_x$ the leaf of ${\mathcal F}$ through $x$. 
If $\nu(x)$ is small, 
the metric is on a ball of radius $\simeq\frac{1}{2\nu(x)}$ around $x$ 
very close (in a topology of large smoothness degree)
to the standard round cylinder.  
Furthermore,
the foliation ${\mathcal F}$ is on this ball very close to the foliation 
of the standard round cylinder by totally-geodesic cross-sectional 2--spheres.
It follows that 
the $\rho(H)$--invariant metric $g|_{{\mathcal F}_x}$
is close, in terms of $\nu(x)$ and $|G|$,
to a $\rho(H)$--invariant round metric 
with scalar curvature $\equiv1$. 
Let $\phi_0\co S^2(\sqrt{2})\to{\mathcal F}_x$ be 
a corresponding almost isometric diffeomorphism 
which is $(\hat\rho_0,\rho|_H)$--equivariant 
with respect to a suitable isometric action 
$\hat\rho_0\co H\acts S^2(\sqrt{2})$. 
Using the $\rho$--invariant line field perpendicular to ${\mathcal F}$
and its integral lines,
we can extend $\phi_0$ to a 
$(\hat\rho,\rho|_H)$--equivariant embedding 
$\phi\co S^2(\sqrt{2})\times(-\frac{1}{2\nu(x)},\frac{1}{2\nu(x)})\embed M$,
$\phi|_{S^2(\sqrt{2})\times\{0\}}=\phi_0$, 
where $\hat\rho$ is a suitable extension of $\hat\rho_0$
to an isometric action of $H$ on 
$S^2(\sqrt{2})\times(-\frac{1}{2\nu(x)},\frac{1}{2\nu(x)})$. 
Given $\eps>0$, 
the restriction of $\phi$ to 
$S^2(\sqrt{2})\times(-\frac{1}{\eps},\frac{1}{\eps})$
is arbitrarily close to an isometry
provided that $\nu(x)$ is sufficiently small. 
\end{proof}

\subsection{Neck--cap geometry}
\label{sec:cap}

Let $(M^3,g)$ be a connected orientable closed Riemannian 3--manifold. 

Let $\eps>>\eps_1>0$. 
Let $A_0(\eps_1)\subseteq M$ be the open subset of points $x$
such that $(M,x,g)$ is $\eps_1$--approximated 
by a $\kappa$-- or a standard solution $(N,x',h)$. 
Let $A_1(\eps,\eps_1):=A_0(\eps_1)\cap\{\rrad(\cdot,M)>D(\eps)\}$,
with $D(\eps)$ to be specified in Proposition \ref{caps} below.
Note that for $x\in A_1$, a $\kappa$--solution $\eps_1$--approximating 
$(M,x,g)$ is a $\kappa_0$--solution.

Suppose that $x\in A_0$.
If $\eps_1$ is sufficiently small (in terms of $\eps$), 
then centers of $\frac{\eps}{2}$--necks in $N$
correspond via the approximation 
to centers of $\eps$--necks in $M$. 
In particular, 
if $x'\in N_{\eps/2}^{\neck}$ then $x\in M_{\eps}^{\neck}$,
respectively, 
if $x\in M_{\eps}^{\nn}$ then $x'\in N_{\eps/2}^{\nn}$.  

The neck--cap alternative carries over from the local models 
to regions which are well approximated by them. 
We obtain from Proposition~\ref{capsinmodel}: 
\begin{prop}[Caps]
\label{caps}
For any sufficiently small $\eps>0$ 
there exist constants 
$D(\eps)>>\bar d(\eps)>d(\eps)>0$
and $0<\epsseven(\eps)\leq\frac{1}{2D}$ such that the following hold:

(i) If $0<\eps_1\leq\epsseven$
and $x\in M_{\eps}^{\nn}\cap A_1$, 
then there exists an $(\eps,d)$--cap $C$ centered at $x$.
It satisfies $\rrad(x,C)<\bar d$. 
Moreover, 
if ${\mathcal N}\subset C$ is an $\eps$--neck 
representing the end of $C$
as in definition~\ref{def:cap},
then ${\mathcal N}\subset M_{\eps}^{\neck}$. 

(ii) If $\hat C\subset M$ is another $(\eps,d)$--cap 
centered at a point 
$\hat x\in (M_{\eps}^{\nn}\cap A_1)-C$, 
then $\hat C\cap C=\emptyset$. 
\end{prop}
We will refer to an $(\eps,d(\eps))$--cap in $M$ 
simply as an {\em $\eps$--cap}.

\begin{cor}
\label{cor:capsintersect}
If $C_1,C_2$ are $\eps$--caps centered at $x_1,x_2\in M_{\eps}^{\nn}\cap A_1$, 
then: 

$C_1\cap C_2\neq\emptyset\Leftrightarrow 
C_1\cap M_{\eps}^{\nn}\cap A_1=C_2\cap M_{\eps}^{\nn}\cap A_1$
\end{cor}
\begin{proof}
Direction ``$\Leftarrow$'' is trivial. 
We prove direction ``$\Rightarrow$''.
By part (ii) of the proposition, 
we have that $x_2\in C_1$ and $x_1\in C_2$. 
Let $x\in C_1\cap M_{\eps}^{\nn}\cap A_1$
and let $C$ be an $\eps$--cap centered at $x$. 
Then $x_1\in C$. 
Hence $x_1\in C\cap C_2\neq\emptyset$ and therefore also $x\in C_2$.
\end{proof}

Consequently,
for $\eps$--caps centered at points in $M_{\eps}^{\nn}\cap A_1$,
the relation defined by $C_1\sim' C_2$ 
iff $C_1\cap C_2\neq\emptyset$ 
is an equivalence relation. 
Equivalent caps differ only outside $M_{\eps}^{\nn}\cap A_1$, 
and inequivalent caps are disjoint. 
Furthermore, 
the relation on $M_{\eps}^{\nn}\cap A_1$ defined by $x_1\sim x_2$, 
iff there exists an $\eps$--cap $C$ containing $x_1$ and $x_2$,
is an equivalence relation. 
The equivalence class of a point $x$ is given by 
$C\cap M_{\eps}^{\nn}\cap A_1$
for any $\eps$--cap containing $x$. 

Note that there exists $\rho>0$ with the property that 
every $\eps$--cap $C$ centered at $x\in M_{\eps}^{\nn}\cap A_1$
contains $\tildeB(x,\rho)$. 
Consequently, since $M$ is closed, 
there can only be finitely many equivalence classes of $\eps$--caps.

Note that the situation when $A_1\subsetneq A_0$ is very special.
There exists $x\in A_0$ with $\rrad(x,M)\leq D$ 
and hence $M$ is {\em globally} $\eps_1$--approximated 
by a compact $\kappa$--solution,
i.e.\ by a compact $\kappa_0$--solution or a spherical space form. 
We choose the constant $\epsseven$ in Proposition~\ref{caps} 
sufficiently small, 
such that in addition one has 
a uniform positive lower bound for sectional curvature, 
$sec\geq c\scal(x)$ on $M$ with a constant $c(D(\eps))=c(\eps)>0$,
cf.\ the last assertion of Addendum~\ref{add:bounds}. 

\subsection{Equivariant tube--cap decomposition}
\label{sec:eqtc}

We now combine the discussions in section~\ref{sec:fol} 
and the previous section to describe, in the equivariant case, 
the global geometry of the region 
which is well approximated by the local models. 
This comprises the region of sufficiently large positive scalar curvature 
in the time slice of a Ricci flow. 

Let $(M^3,g)$ be again a connected orientable closed Riemannian 3--manifold
and let $\rho\co G\acts M$ be an isometric action by a finite group. 
In the following,
$\eps>0$ denotes a sufficiently small positive constant. 
It determines via Proposition~\ref{caps}
the even much smaller positive constant $\eps_1$. 
All $\eps$--caps are assumed 
to be centered at points in $M_{\eps}^{\nn}\cap A_1(\eps,\eps_1)$. 
Furthermore, 
we suppose that $A_1=A_0$, 
compare the remark in the end of section~\ref{sec:cap}.

Let $C\subset M$ be an $\eps$--cap. 
By the construction of caps, 
the end of $C$ is contained in $M_{\eps}^{\neck}$ 
and hence in the foliated region $F$, 
cf.\ section~\ref{sec:fol}. 
Let $T$ be the connected component of $F$
containing the end of $C$.
We will refer to $T$ as the {\em $\eps$--tube} associated to $C$. 
Of course, $C\not\subset F$, e.g.\ for topological reasons. 
Thus $\D T$ consists of two embedded 2--spheres. 
One boundary sphere $\D_{\mathrm{inn}}T$ of $T$ is contained in 
$C\cap M_{\eps}^{\nn}$,
and the other boundary sphere $\D_{\mathrm{out}}T$ 
is contained in $M_{\eps}^{\nn}-C$. 
(Note that $\D F\subset M_{\eps}^{\nn}$.) 

We consider now two situations 
which are of special interest to us. 

{\em Situation 1: $M=A_1$.}
This will cover the case of extinction 
to be discussed in section~\ref{sec:befext}.

If $M_{\eps}^{\neck}=M$, 
then $M$ is globally foliated
and consequently $M\cong S^2\times S^1$. 

If $M_{\eps}^{\neck}\subsetneq M$, 
let $C\subset M$ be an $\eps$--cap
and let $T$ be the $\eps$--tube associated to $C$. 
The boundary sphere $\D_{\mathrm{out}}T$ 
is contained in a different $\eps$--cap $\hat C$. 
The caps $C$ and $\hat C$ are disjoint, 
and we obtain the tube--cap decomposition 
\begin{equation}
\label{tcsit1}
M=C\cup T\cup\hat C 
\end{equation} 
of $M$.
In this case, there are exactly two equivalence classes of $\eps$--caps. 

{\em Situation 2.}
This more general situation which we describe now 
is tailored to apply to the highly curved region in a Ricci flow 
short before a surgery time, cf.\ section~\ref{sec:topeff}. 

By a {\em funnel} $Y\subset M_{\eps}^{\neck}$ 
we mean a submanifold $\cong S^2\times [0,1]$
which is a union of leaves of the foliation ${\mathcal F}$,
and which has one highly curved boundary sphere $\D_hY$ 
and one boundary sphere $\D_lY$ with lower curvature.
Quantitatively,
we require that 
$\min\scal|_{\D_hY} > C\cdot\max\scal|_{\D_lY}$, 
where $C(\eps)>>1$ is a constant 
greater than the bound 
for the possible oscillation of scalar curvature on $\eps$--caps. 
That is, $C$ is chosen as follows, 
cf.\ Addendum~\ref{add:bounds}: 
If $y_1,y_2$ are points in any $\eps$--cap then 
$\scal(y_1)<C\cdot \scal(y_2)$. 

Suppose that we are given a finite $\rho$--invariant family 
of pairwise disjoint funnels $Y_j\subset M_{\eps}^{\neck}$ 
(corresponding later to parts of horns),
that $M_1\subset M$ is a union of components of 
$M-\cup_jY_j$ such that $\D M_1=\cup_j\D_hY_j$, 
and furthermore that $M_1\subset A_1$. 

We restrict our attention to those $\eps$--caps which intersect $M_1$. 
Under our assumptions,
every such cap $C$ is contained in $M_1\cup Y$, $Y:=\cup_jY_j$,
and all $\eps$--caps equivalent to $C$ also intersect $M_1$. 

Let $C_1,\dots, C_m$ be representatives for the equivalence classes 
of these caps, $m\geq0$. 
We have that 
$C_i\cap M_{\eps}^{\nn}\subset A_1$. 
Let $T'_i$ be the $\eps$--tube associated to $C_i$.
It is the unique component of $F$
such that every $\eps$--cap equivalent to $C_i$ 
contains precisely one of its boundary spheres, 
namely $\D_{\mathrm{inn}}T'_i:=\D T'_i\cap C_i$. 
In particular, $T'_i$ depends only on the equivalence class $[C_i]$. 

If $T'_i\cap Y\neq\emptyset$,  
then we truncate $T'_i$ where it leaves $\ol{M_1}$. 
That is, we replace $T'_i$ by the compact subtube $T_i\subseteq T'_i$ 
with the properties that 
$\D_{\mathrm{inn}}T_i=\D_{\mathrm{inn}}T'_i$ 
and $\D_{\mathrm{out}}T_i=T_i\cap Y$ 
is a sphere component of $\D_hY$. 
Otherwise, we put $T_i:=T'_i$. 

If $T_i\cap\D Y=\emptyset$, 
then $T_i\subset M_1$ 
and $\D_{\mathrm{out}}T_i$ is contained in a different cap 
$C_{\iota(i)}$, $\iota(i)\neq i$. 
As in situation 1 above,  
\begin{equation}
\label{tcsit2a}
C_i\cup T_i\cup C_{\iota(i)}
\end{equation} 
is a closed connected component of $M_1$. 
On the other hand, 
if $T_i\cap\D Y\neq\emptyset$
and hence 
$T_i\cap\D Y=\D_{\mathrm{out}}T_i$,
then 
\begin{equation}
\label{tcsit2b}
C_i\cup T_i
\end{equation}
is a component of $\ol{M_1}$
with one boundary sphere. 
We call it an {\em $\eps$--tentacle}. 
All caps $C_i$ occur in a component (\ref{tcsit2a}) or (\ref{tcsit2b}).

There may be further components of $M_1$ 
that are contained in $F$.
They are either closed and $\cong S^2\times S^1$, 
or they are {\em tubes} $\cong S^2\times[0,1]$ 
whose boundary spheres are components of $\D_h Y$. 

This provides the tube--cap decomposition of $M_1$. 

{\em Equivariance.}
In situation 2,
the tube $T_i$ depends only on the equivalence class $[C_i]$.
Therefore the subgroup $\Stab_G([C_i])$ of $G$
preserves $T_i$,  
every ${\mathcal F}$--leaf contained in $T_i$ 
and hence also the cap $C_i$. 
The union of tubes $T_i$ is invariant under the whole group $G$,  
and the caps $C_i$ can be adjusted such that 
their union is $\rho$--invariant, too. 
Situation 1 is analogous to the case (\ref{tcsit2a}) of situation 2.
Thus the tube-neck decomposition can also always be done equivariantly. 

\section{Actions on caps}
\label{sec:actcap}

In order to compare the $G$--actions before and after surgery,
see section~\ref{sec:topeff} below, 
one needs to classify the action on the highly curved region
near the singularity of the Ricci flow,
i.e.\ short before the surgery. 
On the necklike part,
one has very precise control on the geometry and hence also on the action. 
On the other hand, 
the caps are diffeomorphic to $B^3$ or $\R P^3-\bar B^3$ 
but at least in the first case
there is relatively little information about their geometry. 
However,
caps may have nontrivial stabilizers in $G$ 
and one must verify that their actions on the caps are standard. 
This is the aim of the present chapter. 
For an alternative way of controlling the action on caps (which does not
require Proposition \ref{subact}) we refer to \cite{diss}.

Since we are working with non-equivariant approximations by local models,
we obtain almost isometric actions on local models 
which are not defined everywhere but only on a large region 
whose size depends on the quality of the approximation. 
The key step,
see section~\ref{sec:approx}, 
is to approximate such partially defined almost isometric actions 
by globally defined isometric actions on nearby local models with symmetries.
This is possible due to the compactness properties 
of the spaces of local models. 
We then verify in section~\ref{sec:isomactmod}
that isometric actions on local models are standard,
as well as their restrictions to invariant caps. 
From this we deduce in section~\ref{sec:capactstand} 
that the actions on caps are indeed standard. 

\subsection{Approximating almost isometric actions on local models
by isometric ones}
\label{sec:approx}

Suppose that $\rho\co G\acts (M,g)$ is an isometric action, 
that $(N,h)$ is a local model
with normalized curvature in a base point, $\scal(x')=1$, 
and that 
$\phi\co \tildeB^N(x',\frac{1}{\eps_1})\to M$ is an
$\eps_1$--homothetic embedding 
whose image contains an open subset $V$ 
preserved by a subgroup $H\leq G$.
Then the pulled-back action 
$\phi^{*}\rho|_H\co H\acts\phi^{-1}(V)$
is $\tilde\eps_1(\eps_1)$--isometric 
with $\tilde\eps_1(\eps_1)\to 0$ for $\eps_1\to 0$, 
compare the definitions in section~\ref{sec:defnot}. 

The next result allows 
to improve approximations by almost isometric partial actions 
to approximations by global isometric actions. 

\begin{lem}
\label{equivarcaps}
For $a,\zeta>0$ and a finite group $H$
there exists $\eta(a,\zeta,H)>0$
such that: 

Suppose that $(N,x')$ is a time slice 
of a $\kappa_0$-- or a rescaled standard solution,
normalized so that $\scal(x')=1$.
Furthermore, for some $\eps_1\in(0,\eta]$ 
let $\rho\co H\acts V$ be an $\eps_1$--isometric action
on an open subset $V$ with 
$\tildeB(x',\frac{99}{100}\frac{1}{\eps_1})\subseteq V\subseteq N$, 
and suppose that there is a $\rho$--invariant open subset $A$, 
$x'\in A\subset V$,  
with  $\rrad(x',A)<a$. 

Then there exists a time slice $(\hat N,\hat x')$ 
of a $\kappa_0$-- or rescaled standard solution 
which is $\zeta$--close to $(N,x')$, 
a globally defined isometric action $\hat\rho\co H\acts\hat N$ 
and a $(\rho,\hat\rho)$--equivariant smooth embedding 
$\iota\co A\embed\hat N$. 
\end{lem}

\begin{proof}
We argue by contradiction. 
We assume that for some $a,\zeta,H$ 
there exists no such $\eta$. 
Then there exist sequences 
of positive numbers $\eps_{1i}\searrow 0$, 
of time slices of $\kappa_0$-- or standard solutions $(N_i,x'_i)$ with $\scal(x'_i)=1$,
of $\eps_{1i}$--isometric actions 
$\rho_i\co H\acts V_i$ on open subsets $V_i$,
$\tildeB(x'_i,\frac{99}{100}\frac{1}{\eps_{1i}})\subseteq V_i\subseteq N_i$,
and of $\rho_i(H)$--invariant open neighborhoods $A_i$ of $x'_i$
with $\rrad(x'_i,A_i)<a$, 
which violate the conclusion of the lemma for all $i$.

According to the compactness theorems for $\kappa$-- and standard solutions,
cf.\ section~\ref{sec:locmod}, 
after passing to a subsequence,
the $(N_i,x'_i)$ converge smoothly to a time slice $(N_{\infty},x'_{\infty})$
of a $\kappa_0$-- or a renormalized standard solution. 
Hence for $i$ sufficiently large, 
$(N_i,x'_i)$ is $\zeta$--close to $(N_{\infty},x'_{\infty})$.

The convergence of the actions follows: 
By the definition of closeness, 
given $\nu>0$, we have 
for $i\geq i(\nu)$ 
an $\nu$--isometric embedding
$\psi_i\co (\tildeB(x'_{\infty},\frac{1}{\nu}),x'_{\infty})\embed (N_i,x'_i)$ 
such that $\mathrm{im}(\psi_i)\subset V_i$. 
Our assumption implies $\rrad(x'_i,\rho_i(H)x'_i)<a$ and 
that there exists an open $\rho_i(H)$--invariant subset $U_i$
with 
$\tildeB(x'_i,\frac{1}{\nu}-\frac{101}{100}a)\subset U_i\subset \mathrm{im}(\psi_i)$. 
We pull the $\eps_{1i}$--isometric action $\rho_i|_{U_i}$ back 
to an $\tilde\eps_i(\eps_{1i},\nu)$--isometric action 
$\psi_i^*\rho_i\co H\acts\psi_i^{-1}(U_i)$.
Let $(\nu_i)_{i\geq i_0}$ be a sequence of positive numbers, $\nu_i\searrow0$,
such that $i(\nu_i)\leq i$. 
Then the action $\psi_i^*\rho_i$ is in fact 
$\tilde\eps_i(\eps_{1i},\nu_i)$--isometric, 
and $\tilde\eps_i(\eps_{1i},\nu_i)\to0$.
We have that $U_i\nearrow N_{\infty}$
and $\limsup_{i\to\infty}\rrad(x'_{\infty},\psi_i^*\rho_i(H)x'_{\infty})
\leq a$. 
This implies that, after passing to a subsequence,
the actions $\psi_i^*\rho_i$ converge 
to an isometric limit action 
$\rho_{\infty}\co H\acts N_{\infty}$
with $\rrad(x'_{\infty},\rho_{\infty}(H)x'_{\infty})\leq a$. 

Consider now the open subsets $\psi_i^{-1}(A_i)\subset N_{\infty}$. 
Their diameters are uniformly bounded, 
$\rrad(x'_{\infty},\psi_i^{-1}(A_i))<\frac{101}{100}a$. 
For $i\to\infty$, the almost isometric action $\psi_i^*\rho_i$
and the isometric action $\rho_{\infty}$
become arbitrarily close on $\psi_i^{-1}(A_i)$
(actually on a much larger subset). 
Following Palais \cite{Palais} and Grove--Karcher \cite{GroveKarcher}, 
we construct for large $i$ smooth maps conjugating 
$\psi_i^*\rho_i|_{\psi_i^{-1}(A_i)}$ into $\rho_{\infty}$. 
For $h\in H$ the smooth maps 
$\phi_{i,h}:=\rho_{\infty}(h)^{-1}\circ(\psi_i^*\rho_i)(h)\co 
\tildeB(x'_{\infty},\frac{101}{100}a)\to N_{\infty}$ 
converge to the identity. 
For $i$ sufficiently large, the sets $\{\phi_{i,h}(x):h\in H\}$,  
$x\in B(x'_{\infty},\frac{101}{100}a)$, 
have sufficiently small diameter 
so that their center of mass $c_i(x)$ is well-defined, 
see \cite[Proposition 3.1]{GroveKarcher}. 
The maps $c_i$ are smooth, cf.\ \cite[Proposition 3.7]{GroveKarcher}, 
and $c_i\to \id_{N_{\infty}}$. 
Note that the center of the set 
$\{(\rho_{\infty}(h')^{-1}\circ(\psi_i^*\rho_i)(h')
\circ(\psi_i^*\rho_i)(h))(x):h'\in H\}$
equals $c_i((\psi_i^*\rho_i)(h)(x))$.
On the other hand, 
it is the $\rho_{\infty}(h)$--image of the center of the set 
$\{(\rho_{\infty}(h'h)^{-1}\circ(\psi_i^*\rho_i)(h'h))(x):h'\in H\}$,
and the latter equals $c_i(x)$.
So $\rho_{\infty}(h)\circ c_i=c_i\circ(\psi_i^*\rho_i)(h)$ on 
$B(x'_{\infty},\frac{101}{100}a)$
and thus 
$\rho_{\infty}(h)\circ(c_i\circ\psi_i^{-1})=
(c_i\circ\psi_i^{-1})\circ\rho_{i}(h)$
on $A_i$.
The existence of the 
$(\rho_i,\rho_{\infty})$--equivariant smooth embeddings 
$\iota_i=c_i\circ\psi_i^{-1}\co A_i\embed N_{\infty}$
shows that the conclusion of the lemma is satisfied for large $i$.
Putting $(\hat N,\hat x')=(N_{\infty},x'_{\infty})$
and $\hat\rho=\rho_{\infty}$, 
we obtain a contradiction.  
\end{proof}

\subsection{Isometric actions on local models are standard}
\label{sec:isomactmod}

Let $(N,h)$ be a local model,
i.e.\ a time slice of a $\kappa$-- or standard solution,
and let $\rho\co H\acts N$ be an isometric action by a finite group. 

If $N$ is the time slice of a standard solution,
it has rotational symmetry,
$\Isom(N,h)\cong O(3)$. 
The natural action $\Isom(N,h)\acts N$
is smoothly conjugate to the orthogonal action $O(3)\acts B^3$, 
and in particular 
the action $\rho$ is standard. 

If $N$ is a round cylinder 
or its orientable smooth $\Z_2$--quotient, 
then $\rho$ is also clearly standard. 

Otherwise, 
$N$ is the time slice of a $\kappa$--solution
and has strictly positive sectional curvature.
This case is covered by the following result. 

\begin{prop}
\label{isomactkappa}
An isometric action by a finite group on a complete 3--manifold
with strictly positive sectional curvature 
is smoothly conjugate to an 

(i) 
orthogonal action on $\R^3$ if the manifold is noncompact.  

(ii)
isometric action on a spherical space form if the manifold is compact. 
\end{prop}

The compact case is a direct consequence of Hamilton \cite{Hamilton_posric}.
The noncompact case follows from an equivariant version of the Soul Theorem
which holds in all dimensions: 

\begin{prop}
\label{equivsoul}
Let $W^n$ be a complete noncompact Riemannian manifold 
with strictly positive sectional curvature. 
Suppose that $H$ is a finite group 
and that $\rho\co H\acts W$ is an isometric action.  
Then $H$ fixes a point and $\rho$ is smoothly conjugate 
to an orthogonal action on $\R^n$. 
\end{prop}
\begin{proof}
One follows the usual proof of the Soul Theorem by Gromoll--Meyer
\cite{poscurv}, Cheeger--Gromoll \cite{soul},
see e.g.\ the nice presentation in Meyer \cite[Chapters 3.2 and 3.6]{Meyer}, 
in the special case of strictly positive curvature 
and makes all constructions group invariant. 

In a bit more detail: 
Starting from the collection of all geodesic rays 
with initial points in a fixed $H$--orbit,
one constructs 
an exhaustion $(C_t)_{t\geq0}$ of $W$ 
by $H$--invariant compact totally convex subsets.
We may assume that $C_0$ has nonempty interior. 
Since the sectional curvature is strictly positive,
$C_0$ contains a unique point $s$ at maximal distance from its boundary. 
It is fixed by $H$ and it is a soul for $M$. 
The distance function $d(s,\cdot)$ has no critical points 
(in the sense of Grove and Shiohama \cite{GroveShiohama})
besides $s$. 
One can construct an $H$--invariant gradient-like vector field $X$
for $d(s,\cdot)$ on $W-\{s\}$. 
It can be arranged that $X$ coincides with the radial vector field 
$\nabla d(s,\cdot)$ near $s$. 
Using the flow of $X$ one obtains a smooth conjugacy between $\rho$ 
and its induced orthogonal action $d\rho_s$ on $T_sW\cong\R^n$. 
Near $s$ it is given by the (inverse of the) exponential map. 
\end{proof}

\begin{rem}
In the compact case, 
if the local model $N$ has large diameter, 
then the possibilities for the actions are more restricted,
as the discussion in section~\ref{sec:eqtc} shows.
There is a $\rho$--invariant tube--cap decomposition
$N=C_1\cup T\cup C_2$
and the central leaf $\Si$ of the tube $T$ 
is preserved by $\rho$.
\end{rem}

We now apply Proposition~\ref{subact} 
to deduce that for any cap in a local model, 
which is invariant under an isometric action,  
the restricted action on the cap is standard. 

\begin{cor}
\label{actballinkappa}
Suppose that $\bar C\subset N$ 
is a compact $\rho$--invariant submanifold 
diffeomorphic to $\bar B^3$ or $\R P^3-B^3$. 
Then the restricted action $\rho|_{\bar C}$ is standard.
\end{cor}
\begin{proof}
When $N$ is noncompact,
it can be compactified by adding one or two points 
to a smooth manifold $\cong S^3$ or $\R P^3$,
and the action $\rho$ can be extended to a smooth standard action. 
The latter is clear 
when $N$ is the time slice of a standard solution
or when $N$ is isometric to $S^2\times\R$ or $S^2\times_{\Z_2}\R$.
It follows from Proposition~\ref{equivsoul} when 
$N$ has strictly positive sectional curvature. 
In view of Proposition~\ref{isomactkappa} (ii), 
we may therefore assume that $N$ is {\em metrically}
a spherical space form $S^3/\Ga$. 

Suppose first that $\bar C$ is a ball.
Then $\bar C$ can be lifted to a closed ball 
$\bar B\subset S^3$
and $\rho$ can be lifted to an isometric action $\tilde\rho\co H\acts S^3$
preserving $\bar B$. 
Now Proposition~\ref{subact} implies that the restricted action
$\tilde\rho|_{\bar B}$ is standard, 
and therefore also $\rho|_{\bar C}$. 

We are left with the case when $\bar C\cong\R P^3-B^3$. 
Since $S^3/\Ga$ is irreducible,
the 2--sphere $\D C$ bounds on the other side a ball $B'$,  
and hence $N\cong\R P^3$. 
As before, 
Proposition~\ref{subact} implies that 
$\rho|_{\bar B'}$ is standard. 
It follows that $\D B'=\D C$ can be $\rho$--equivariantly isotoped 
to a (small) round sphere,
and that also the action $\rho|_{\bar C}$ is standard. 
\end{proof}

\subsection{Actions on caps are standard}
\label{sec:capactstand}

We take up the discussion of the equivariant tube--cap decomposition
from section~\ref{sec:eqtc}. 
Let $x\in M_{\eps}^{\nn}\cap A_1$
and let $C\subset M$ be an $\eps$--cap centered at $x$ 
as given by Proposition~\ref{caps}. 
Every other $\eps$--cap that
intersects $C$ agrees with $C$ on $M_{\eps}^{\nn}\cap A_1$, 
cf.\ Corollary~\ref{cor:capsintersect}. 
Let 
$H:=\Stab_{G}(C \cap M_{\eps}^{\nn} \cap A_1)$. 
Then $C$ can be modified to be $H$--invariant, 
and we have that 
$\ga C\cap C=\emptyset$ for all $\ga\in G-H$. 

\begin{prop}
\label{actcaps} 
There exists $0<\epseight(H,\eps)\leq\epsseven$ such
that for $0<\eps_1\leq\epseight$ holds: Let $x\in M_{\eps}^{\neck}\cap
A_1$ be center of an $H$--invariant $\epsilon$--cap $C$,
$H=\Stab_{G} (C \cap M_{\eps}^{\nn} \cap A_1)$, then the restricted
action $\rho|_{\bar C} \co  H\acts\bar C$ is standard.
\end{prop}
\begin{proof} 
Let $(N,x',h)$ be a time slice of a 
$\kappa_0$-- or rescaled standard solution,
normalized so that $\scal(x')=1$. 
Suppose that $(N,x',h)$ 
$\eps_1$--approximates $(M,x,g)$,  
and let $\phi\co \tildeB(x',\frac{1}{\eps_1})\embed M$
be an $\eps_1$--homothetic embedding 
realizing the approximation.

Since $\rrad(x,\rho(H)x)\leq\rrad(x,C)<\bar d$,
cf.\ Proposition~\ref{caps}, 
there exists an open subset $V$ of $N$, 
$\tildeB(x',\frac{1}{\eps_1}-\frac{101}{100}\bar d)\subset V
\subset \tildeB(x',\frac{1}{\eps_1})$,
such that $\phi(V)$ is $\rho(H)$--invariant. 
The pulled-back action $\phi^*\rho$ on $V$ is 
$\tilde\eps_1(\eps_1)$--isometric
with $\tilde\eps_1(\eps_1)\to0$ as $\eps_1\to0$. 
Let $A\subset \tildeB(x',\frac{101}{100}\bar d)$ 
be a $(\phi^*\rho)(H)$--invariant open neighborhood of $x'$
such that $C\subset\phi(A)$. 

Now we can apply Lemma~\ref{equivarcaps}. 
We fix some $\zeta_0>0$ (which will not play a role afterwards) 
and choose 
$\epseight\in(0,\epsseven]$ sufficiently small 
such that 
$0<\eps_1\leq\epseight$ implies 
$\tilde\eps_1(\eps_1)\leq\eta(\frac{101}{100}\bar d,\zeta_0,H)
=:\eta(H,\eps)$.
The lemma yields a time slice $(\hat N,\hat x')$
of a $\kappa_0$-- or rescaled standard solution, 
an isometric action $\hat\rho\co H\acts\hat N$,
and a $(\phi^*\rho,\hat\rho)$--equivariant embedding 
$\iota\co A\embed\hat N$.
The latter implies 
that the action 
$\phi^*\rho|_{\phi^{-1}(\bar C)}$
is smoothly conjugate to the action 
$\hat\rho|_{\iota\circ\phi^{-1}(\bar C)}$,
and thus 
$\rho|_{\bar C}$
is smoothly conjugate to 
$\hat\rho|_{\iota\circ\phi^{-1}(\bar C)}$. 
According to Corollary~\ref{actballinkappa},
the latter action is standard.
\end{proof}

We will henceforth put 
$\epsnine(G,\eps):=
\min\{\epseight(H,\eps):H\leq G\}$ 
and assume that $0<\eps_1\leq\epsnine$.

\section{Equivariant Ricci flow with cutoff and applications}
\label{sec:cutoff}

We will now derive our main results 
about smooth actions by finite groups on closed 3--manifolds.
Given an action 
$\rho_0\co G\acts M_0$,
we choose a $\rho_0$--invariant Riemannian metric $g_0$ on $M_0$.
Perelman's construction of Ricci flow with cutoff 
carries over to the equivariant case
in a straight-forward manner, 
see section~\ref{sec:ex}, 
and yields an equivariant Ricci flow with cutoff 
defined for all times 
and with initial time slice $\rho_0\co G\acts(M_0,g_0)$. 
Based on the fact proven in section~\ref{sec:actcap}
that actions on caps 
in the highly curved regions near the singularities of the Ricci flow
are standard,
we are able to classify the actions on the time slices short before extinction 
and, more generally, 
describe the change of the actions when crossing a singular time. 
We then focus on the case when the initial manifold $M_0$ is irreducible 
and further on actions on elliptic and hyperbolic 3--manifolds. 
Finally,
we discuss actions on $S^2\times\R$--manifolds.

\subsection{Existence} 
\label{sec:ex}

The construction of a Ricci flow with $(r,\de)$--cutoff,
i.e.\ with a specified way of surgery,  
for closed orientable 3--manifolds with arbitrary initial metrics 
is one of the fundamental contributions of Perelman \cite{Perelman_surgery}. 
For a detailed discussion of Ricci flows with surgery 
we refer to \cite[Sections 68--80]{KL}, 
\cite[Chapters 13--17]{MorganTian_poinc} and \cite[Chapter 7]{Bamler}.
We will adopt the notation in \cite[Section 68]{KL}. 
There is the following difference, however. 
We use the parameter $\eps_1$ (instead of $\eps$) 
to measure the quality of approximation
of canonical neighborhoods by local models 
($\kappa_0$--solutions, shrinking spherical space forms or standard solutions).
That is, our parameter $\eps_1$ plays the role of the parameter $\eps$
in \cite[Lemma 59.7 and Definition 69.1]{KL}.

For a Ricci flow with surgery 
$({\mathcal M},(g_t)_{0\leq t<+\infty})$
there is a discrete, possibly empty or infinite, 
sequence of {\em singular} times
$0<t_1<\dots<t_k<\dots$. 
Let $k_{\max}$
denote the number of singular times,
$0\leq k_{\max}\leq+\infty$. 
We denote by $M_k$ the orientable closed 3--manifold 
underlying the time slices  
${\mathcal M}_{t_k}^+$ and ${\mathcal M}_t$ for $t_k<t<t_{k+1}$.  
(We put $t_0:=0$ and,
if $k_{\max}<+\infty$,
also $t_{k_{\max}+1}:=+\infty$.)

We will only consider Ricci flows with {\em $(r,\de)$--cutoff}.
These are Ricci flows with surgery 
where the surgery is performed in a specific way.
Let us recall 
how one passes at a singular time $t_k$ 
from the backward (pre-surgery) time slice ${\mathcal M}_{t_k}^-$
to the forward (post-surgery) time slice ${\mathcal M}_{t_k}^+$,
compare \cite[Definition 73.1]{KL}.
The manifold underlying ${\mathcal M}_{t_k}^-$
is the open subset 
$\Om=\{x\in M_{k-1} \,|\,\limsup_{t\nearrow
t_k}\left|\Rm(x,t)\right|<\infty\}$, where $\Rm$ is the Riemann
curvature tensor.
On $\Om$ the Riemannian metrics converge smoothly to a limit metric,
$g_t\to g_{t_k}^-$ as $t\nearrow t_k$. 
Let $\rho:=\de(t_k)r(t_k)$. 
In order to obtain ${\mathcal M}_{t_k}^+$,
we discard all components of $\Om$ 
that do not intersect 
$\Om_{\rho}=\{x\in \Om \,|\, \scal(x,t_k)\leq \rho^{-2} \}$. 
We say that a component of $M_{k-1}$ 
that does not intersect $\Om_{\rho}$ 
goes {\em extinct} at time $t_k$.
For volume reasons, 
there are only finitely many components $\Om_i$ of $\Om$ 
which do intersect $\Om_{\rho}$. 
If a component $\Om_i$ is closed 
then the Ricci flow smoothly extends to times after $t_k$
and $\Om_i$ survives to $M_k$ without being affected by the surgery. 
Each noncompact component $\Om_i$ has finitely many ends,
and the ends are represented by $\eps_1$--horns 
${\mathcal H}_{ij}\subset\Om_i$. 
The horns are disjoint and contained in the foliated necklike region $F$
introduced in section~\ref{sec:fol}.  
The surgery is performed at $\de(t_k)$--necks 
${\mathcal N}_{ij}\subset{\mathcal H}_{ij}$
which are centered at points with scalar curvature 
$h(t_k)^{-2}$. 
The quantity $h(t_k)<\rho$ is given by \cite[Lemma 71.1]{KL}.
(The necks ${\mathcal N}_{ij}$ are in fact final time slices of strong 
$\de(t_k)$--necks.) 
The horn ${\mathcal H}_{ij}$ is cut along the 
(with respect to the neck parametrization) 
central ${\mathcal F}$--leaf, 
i.e.\ cross-sectional sphere $S_{ij}\subset{\mathcal N}_{ij}$
and capped off by attaching a 3--ball. 
The region $X:={\mathcal M}_{t_k}^-\cap{\mathcal M}_{t_k}^+$
common to backward and forward time slice 
is a compact 3--manifold with boundary $\D X$
equal to the union of the surgery spheres $S_{ij}$. 
One may regard $X$ as a submanifold of both $M_{k-1}$ and $M_k$. 

Let now $G$ be a finite group. 
A {\em $G$--equivariant Ricci flow with surgery}
consists of a Ricci flow with surgery 
$({\mathcal M},(g_t)_{0\leq t<+\infty})$
together with a smooth group action 
$\rho\co G\acts{\mathcal M}$
such that $\rho$ preserves each time slice ${\mathcal M}_t^{\pm}$
and acts on it isometrically.
Moreover, 
we require that $\rho$ maps static curves to static curves. 
The restriction of $\rho$ to the time slab 
${\mathcal M}_{(t_k,t_{k+1})}$
corresponds to a smooth action $\rho_k\co G\acts M_k$
which is isometric with respect to the Riemannian metrics 
$g_{t_k}^+$ and $g_t$ for $t_k<t<t_{k+1}$. 

In \cite{Perelman_surgery}, 
Perelman only discussed the nonequivariant case 
(i.e.\ when $G$ is trivial), 
but not much has to be modified 
to extend the discussion to the equivariant case. 
The usual Ricci flow without surgery on a closed 3--manifold 
preserves the symmetries of the initial metric
(as a consequence of its uniqueness).
Hence the metrics $g_t$, $t_k<t<t_{k+1}$,  
will have the same symmetries as $g_{t_k}^+$. 
To obtain an equivariant Ricci flow with cutoff
for a given equivariant initial condition, 
one must only ensure 
that no symmetries get lost in the surgery process.
Once the surgery necks are chosen equivariantly
in the sense of Lemma~\ref{equivneck},
the surgery process as described in \cite[Section 72]{KL}
does preserve the existing symmetries. 

The equivariant choice of surgery necks can be easily achieved. 
One can arrange that every $\eps_1$--horn ${\mathcal H}_{ij}$
at the surgery time $t_k$  
is saturated with respect to the foliation ${\mathcal F}$
and that the union of the horns is $\rho$--invariant,
i.e.\ the horns are permuted by the group action. 
Note that $H_{ij}=\Stab_G({\mathcal H}_{ij})$ 
preserves every ${\mathcal F}$--leaf 
in ${\mathcal H}_{ij}$. 
The $\de(t_k)$--necks ${\mathcal N}_{ij}$ can also be chosen 
$\rho$--equivariantly and as ${\mathcal F}$--saturated subsets,
and then $H_{ij}=\Stab_G({\mathcal N}_{ij})$. 
Using Lemma~\ref{equivneck}, 
we $\rho$--equivariantly replace 
the ${\mathcal N}_{ij}$ by 
$(H_{ij},\tilde\de(\de(t_k),G))$--necks $\widetilde{\mathcal N}_{ij}$ 
centered at the same points. 
The $\widetilde{\mathcal N}_{ij}$
have the additional property
that the approximating $\tilde\de$--homothetic diffeomorphisms 
$\phi_{ij}\co S^2(\sqrt{2})\times(-\frac{1}{\tilde\de},\frac{1}{\tilde\de})
\to\widetilde{\mathcal N}_{ij}$
can be chosen such that 
the pulled-back actions 
$\hat\rho_{ij}=\phi_{ij}^*(\rho|_{H_{ij}})$ 
on $S^2(\sqrt{2})\times(-\frac{1}{\tilde\de},\frac{1}{\tilde\de})$
are isometric
(and trivial on the interval factor). 
Since $\tilde\de(\de,G)\to0$ as $\de\to0$, 
we can keep $\tilde\de$ arbitrarily small by suitably decreasing $\de$. 

To glue in the surgery caps,
we follow the interpolation procedure 
of \cite[Lemma 72.24]{KL}.
The surgery caps are truncated standard solutions
and have the full $O(3)$--symmetry. 
The gluing can therefore be done equivariantly 
and so $\rho$ extends to an isometric action 
on the glued in surgery caps,
that is, on the entire time slice ${\mathcal M}_{t_k}^+$. 
The time-$t_k$ Hamilton-Ivey pinching condition 
is satisfied on ${\mathcal M}_{t_k}^+$
if $\de$ is chosen sufficiently small.  

The justification of the a priori conditions, 
i.e.\ the argument that for a suitable choice of the parameter functions 
$r,\de\co [0,\infty)\to(0,\infty)$ 
the canonical neighborhood condition 
as formulated in \cite[Section 69]{KL}   
remains valid during the flow, 
is not affected by the presence of a group action, 
compare \cite[Section 77]{KL}. 

One concludes,  
cf.\ \cite[Proposition 5.1]{Perelman_surgery}
and \cite[Proposition 77.2]{KL}, 
that there exists $\epsten(G)>0$ 
such that the following holds:
If $0<\eps_1\leq\epsten$, 
then there exist positive nonincreasing functions 
$r,\bar\de\co [0,\infty)\to(0,\infty)$ 
such that for any normalized initial data 
$\rho\co G\acts{\mathcal M}_0$ 
and any nonincreasing function $\de\co [0,\infty)\to(0,\infty)$ 
with $\de<\bar\de$,  
the $G$--equivariant Ricci flow with $(r,\de)$--cutoff 
is defined for all times. 

We will henceforth put 
$\epseleven(G,\eps):=
\min(\epsten(G),\epsnine(G,\eps))$
and assume that $0<\eps_1\leq\epseleven$.

\subsection{Standard actions short before extinction}
\label{sec:befext}

Let $\rho\co G\acts{\mathcal M}$ be an equivariant Ricci flow with cutoff.
We consider now the situation when 
some of the connected components of $M_{k-1}$ 
go extinct at the singular time $t_k$. 
It is known, cf.\ \cite[Section 67]{KL},
that each such component 
is diffeomorphic to a spherical space form, 
to $\R P^3\sharp\,\R P^3$ or to $S^2\times S^1$. 

\begin{thm}[Extinction]
\label{thm:ext}
Suppose that $M_{k-1}^{(1)}$ is a connected component of $M_{k-1}$ 
which goes extinct at the singular time $t_k$.
Then the part 
$\Stab_G(M_{k-1}^{(1)})\acts M_{k-1}^{(1)}$ 
of the action $\rho_{k-1}\co G\acts M_{k-1}$ is standard.
\end{thm}

We recall from section~\ref{defstan} 
that an action on a spherical space form 
is standard if and only if it is smoothly conjugate 
to an isometric action, 
and an action on $S^2\times S^1$ or $\R P^3\sharp\,\R P^3$
is standard 
if and only if there exists an invariant Riemannian metric 
locally isometric to $S^2\times\R$.

\begin{proof}
If $t\in(t_{k-1},t_k)$ is sufficiently close to $t_k$, 
then ${\mathcal M}_t=(M_{k-1}^{(1)},g_t)$ has everywhere high scalar curvature,
$\scal>\frac{99}{100}r(t_k)^{-2}\de(t_k)^{-2}$,
and is therefore everywhere locally $\eps_1$--approximated by a local model,
i.e.\ we have that $M_{k-1}^{(1)}=A_0$. 
(We adopt the notation of section~\ref{sec:tcdec}
with $(M_{k-1}^{(1)},g_t)$ playing the role of $(M,g)$.)

If there exists $x\in M_{k-1}^{(1)}$ with $\rrad_t(x,M_{k-1}^{(1)})\leq D$, 
then $(M_{k-1}^{(1)},g_t)$ has strictly positive sectional curvature,
compare (the end of) section~\ref{sec:cap}
and the assertion follows from Hamilton \cite{Hamilton_posric}. 

We may therefore assume that $M_{k-1}^{(1)}=A_1$  
and consider the equivariant tube--cap decomposition 
as in situation 1 of section~\ref{sec:eqtc}. 
If $F=M_{k-1}^{(1)}$, 
then $M_{k-1}^{(1)}\cong S^2\times S^1$ 
and Corollary~\ref{acts2s1} implies that 
$\Stab_G(M_{k-1}^{(1)})\acts M_{k-1}^{(1)}$ is standard. 
Otherwise, 
the tube--cap decomposition has the form $M_{k-1}^{(1)}=C_1\cup T\cup C_2$ 
as in (\ref{tcsit1}), 
and $M_{k-1}^{(1)}$ is diffeomorphic to 
$S^3$, $\R P^3$ or $\R P^3\sharp\,\R P^3$. 
The $\rho_{k-1}$--action of $\Stab_G(M_{k-1}^{(1)})$
preserves the central ${\mathcal F}$--leaf of $T$
but it may switch the caps.  
The stabilizer $G':=\Stab_G(C_1)=\Stab_G(C_2)$ has index 1 or 2 
in $\Stab_G(M_{k-1}^{(1)})$. 
According to Proposition~\ref{actcaps},
the restriction of the $\rho_{k-1}(G')$--action to each cap $C_i$ is standard. 
With Corollary~\ref{exttoball} we conclude that 
$\Stab_G(M_{k-1}^{(1)})\acts M_{k-1}^{(1)}$ is standard. 
\end{proof}

\subsection{Topological effect of surgery on group actions}
\label{sec:topeff}

Let $\rho\co G\acts{\mathcal M}$ be an equivariant Ricci flow with cutoff.
We describe now in general,
how the actions before and after a surgery time 
are related to each other. 

\begin{thm}[Topological effect of surgery]
\label{thm:topeff}
For each $k\geq1$,
the action $\rho_{k-1}$ before the surgery time $t_k$ 
is obtained from the action $\rho_k$ afterwards 
in three steps:

(i)
First, one takes the disjoint union of $\rho_k$ 
with a standard action on a finite (possibly empty) union 
of $\R P^{3}$'s.
The stabilizer in $G$ of each such $\R P^{3}$ has a fixed point on it. 

(ii)
Then one forms an equivariant connected sum.
The $\R P^{3}$ components mentioned in (i) correspond to ends of the graph 
associated to the connected sum.
(Compare section~\ref{sec:eqcs}.) 

(iii)
Finally,
one takes the disjoint union 
with a standard action on a closed (possibly empty) 3--manifold
whose components are diffeomorphic 
to a spherical space form, 
to $\R P^3\sharp\,\R P^3$ or to $S^2\times S^1$. 
(These are the components going extinct at time $t_k$,
cf.\ Theorem~\ref{thm:ext}.)
\end{thm}

\begin{proof} 
We recall that 
$X={\mathcal M}_{t_k}^-\cap{\mathcal M}_{t_k}^+$
and $\D X$ is the union of the surgery spheres $S_{ij}$. 

When passing from $M_{k-1}$ to $M_k$,
$M_{k-1}-X$ is replaced by a union of balls $B_{ij}$ 
which are attached to the boundary spheres $S_{ij}$ of $X$,
and the restriction of the action $\rho_{k-1}$ to $\ol{M_{k-1}-X}$ 
is replaced by a standard action on the union of the glued-in balls. 

The restriction of $\rho_{k-1}$ to the closed components of $M_{k-1}-X$ 
is standard by Theorem~\ref{thm:ext}. 
This gives step (iii). (Note that the proof follows the
{\em forward} surgery process, so going backwards reverses the
order of the steps.)

Let $Z$ denote the closure of the union of the components of $M_{k-1}-X$ 
with nonempty boundary. 
Then $Z$ is a compact manifold with boundary $\D Z=\cup S_{ij}$. 
In order to analyze $\rho_{k-1}|_Z$, 
we apply the equivariant tube--cap decomposition 
as in situation 2 of section~\ref{sec:eqtc}. 
We choose a $\rho$--invariant family of funnels 
$Y_{ij}\subset{\mathcal H}_{ij}$ 
such that $\D_lY_{ij}=S_{ij}$ with respect to the metric $g_{t_k}^-$. 
That is,
$Y_{ij}$ is contained in the end of $\Om_i$ bounded by $S_{ij}$. 
Let $Y:=\cup Y_{ij}$. 

For a time $t<t_k$
sufficiently close to $t_k$, 
the metric $g_{t_k}^-$ is on the compact manifold $X\cup Y$
arbitrarily well approximated by $g_t$.
Furthermore,  
$\scal(\cdot,t)\geq \frac{99}{100}\rho^{-2}$ on $Z$
and $(Z,g_t)$ 
is therefore everywhere locally $\eps_1$--approximated by a local model,
i.e.\ we have that $Z-Y\subset A_0$. 
Since $Z$ has no closed component,
we have in fact that $Z-Y\subset A_1$. 
According to the discussion in section~\ref{sec:eqtc},
each component of $(Z,g_t)$ is either 
an $\eps_1$--tentacle 
attached to one of the spheres $S_{ij}$ as in (\ref{tcsit2b}),
or an $\eps_1$--tube connecting two of the spheres $S_{ij}$. 

Instead of removing the $\eps_1$--tubes and equivariantly attaching balls 
to the surgery spheres bounding them,
we may cut the $\eps_1$--tubes along their central ${\mathcal F}$--leaves
and attach balls to the resulting boundary spheres. 
The smooth conjugacy type of the action thus obtained is the same. 
Surgery on a tube corresponds to an edge of the graph of the
equivariant connected sum of step (ii).

Tentacles are diffeomorphic to 
$\bar B^3$ or $\R P^3-B^3$, 
and Proposition~\ref{actcaps} implies that the restriction of $\rho_{k-1}$ 
to the union of the $\eps_1$--tentacles is standard. 
We decompose the action $\rho_{k-1}$ 
as an equivariant connected sum 
(in the sense of section~\ref{sec:eqcs}) 
along the family of surgery spheres $S_{ij}$ bounding $\eps_1$--tentacles. 
Each tentacle contributes a summand diffeomorphic to $S^3$ or $\R P^3$
which corresponds to an end of the graph associated 
to the connected sum decomposition. 
In view of Corollary~\ref{exttoball},
the restriction of the action to the union of these summands is standard. 
Vice versa,
when passing from $\rho_k$ back to $\rho_{k-1}$,
the effect of replacing the surgery caps $B_{ij}$
corresponding to $\eps_1$--tentacles by the tentacles
amounts to a connected sum with a standard action
on a union of $S^3$'s and $\R P^3$'s,
and these latter summands correspond to ends of the graph 
associated to the connected sum. 
Adding the $S^3$ summands does not change the smooth conjugacy type 
of the resulting action, 
and they can therefore be omitted.

This completes the proof of the theorem.
\end{proof}

Forgetting about the $G$--action for a moment,
the effect of surgery on the topology of the time slices is as follows.
$M_{k-1}$ is obtained from $M_k$ in two steps: 
First, one takes connected sums of components of $M_k$
and, possibly, copies of $\R P^3$ and $S^2\times S^1$.
Secondly, one takes the disjoint union with finitely many (possibly zero) 
spherical space forms and copies of 
$\R P^3\sharp\,\R P^3$ and $S^2\times S^1$. 
(Note that our definition of equivariant connected sum of an action 
allows connected sums of components with themselves.
Therefore no $S^2\times S^1$ summands are needed in the statement 
of part (i) of Theorem~\ref{thm:topeff}.)

\subsection{The irreducible case}
\label{sec:irred}

Let $\rho\co G\acts{\mathcal M}$ be an equivariant Ricci flow with cutoff 
and suppose now that the initial manifold $M_0$ is irreducible. 
Then only 3--spheres can
split off and the effect of surgery on the group action is more
restricted. 
Theorem~\ref{thm:topeff} specializes to:

\begin{cor}
\label{cor:uniquecomp1}
Suppose that the orientable closed 3--manifold $M_0$ 
is connected and irreducible.

(i)
If $M_0\cong S^3$,
then every manifold $M_k$ is a union of 3--spheres
(possibly empty for $k=k_{\max}$).
The action $\rho_{k-1}$ arises from $\rho_k$ 
by first forming an equivariant connected sum 
and then taking the disjoint union 
with a standard action on a finite union of 3--spheres. 

(ii)
If $M_0\not\cong S^3$,
then there exists $k_0$, $0\leq k_0\leq k_{\max}$, such that: 
For $0\leq k\leq k_0$,  
the manifold $M_k$ has a unique 
connected component $M_k^{(0)}\cong M_0$, 
and all other components are $\cong S^3$. 
For $k_0<k\leq k_{\max}$,
$M_k$ is a union of 3--spheres 
(possibly empty for $k=k_{\max}$). 

Furthermore, 
for $1\leq k\leq k_0$, 
the action $\rho_{k-1}|_{M_{k-1}^{(0)}}$ 
is an equivariant connected sum of $\rho_k|_{M_k'}$
where $M_k'$ is a $\rho_k$--invariant union of 
${M_k^{(0)}}$ with some of the $S^3$--components of $M_k$. 

If $k_0<k_{\max}$, 
then either $M_{k_0}^{(0)}$ goes extinct at time $t_{k_0+1}$
and is diffeomorphic to a spherical space form,
or $M_{k_0}^{(0)}$ does not go extinct at time $t_{k_0+1}$
and is $\cong\R P^3$. 
In the first case,
the action $\rho_{k_0}|_{M_{k_0}^{(0)}}$ is standard.
In the second case,
it is an equivariant connected sum of 
the union of a standard action on $\R P^3$ 
with an action on a finite union of 3--spheres. 
\end{cor}
\begin{proof}
According to Theorem~\ref{thm:topeff}, 
$M_0$ is for every $k$ 
the connected sum of the components of $M_k$ 
and possibly further closed orientable $3$--manifolds
(spherical space forms and copies of $S^2\times S^1$). 
Since $M_0$ is irreducible,
$M_k$ can have at most one component $M_k^{(0)}\not\cong S^3$,
and this component must itself be irreducible. 
If $M_k$ contains such a component,
then so does $M_l$ for $0\leq l\leq k$,
and we have that 
$M_k^{(0)}\cong M_{k-1}^{(0)}\cong\dots\cong M_0^{(0)}=M_0$. 
Let $k_0$, $-1\leq k_0\leq k_{\max}$,
be maximal 
such that $M_k$ has such a component $M_k^{(0)}$ for $0\leq k\leq k_0$. 

(i) 
Here $k_0=-1$
and all components of the $M_k$ are 3--spheres. 
Step (i) in Theorem~\ref{thm:topeff} must be empty,
and the components of step (iii) can only be 3--spheres. 

(ii)
Now $k_0\geq0$.
If $1\leq k\leq k_0$,
then again step (i) in Theorem~\ref{thm:topeff} must be empty,
and the components of step (iii) can only be 3--spheres. 
That is, 
$\rho_{k-1}$ arises from $\rho_k$
by first taking an equivariant connected sum
and then taking the disjoint union with a standard action 
on a finite union of 3--spheres. 
Our assertion for $\rho_{k-1}|_{M_{k-1}^{(0)}}$ follows.

By Theorem~\ref{thm:topeff}, 
a component of $M_k$ which does not go extinct at time $t_{k+1}$
decomposes as the connected sum 
(in the usual non-equivariant sense) 
of some components of $M_{k+1}$ 
and, possibly, copies of $\R P^3$ and $S^2\times S^1$. 
In our situation, 
if $k_0<k_{\max}$
and $M_{k_0}^{(0)}$ does not go extinct at time $t_{k_0+1}$,
then $M_{k_0}^{(0)}$ must be diffeomorphic to $\R P^3$, 
because it is irreducible and $M_{k_0+1}$ is a union of 3--spheres. 
If $M_{k_0}^{(0)}$ goes extinct at time $t_{k_0+1}$,
then it must be a spherical space form,
because it is irreducible,
compare the first paragraph of section~\ref{sec:befext}. 
The claim concerning $\rho_{k_0}|_{M_{k_0}^{(0)}}$ 
follows from Theorem~\ref{thm:topeff}, respectively,
from Theorem~\ref{thm:ext}. 
\end{proof}

Now we use the deep fact 
that on a connected closed orientable 3--manifold with finite fundamental group 
the Ricci flow with cutoff 
goes extinct for any initial metric, see
Perelman \cite{Perelman_extinction}, Colding--Minicozzi
\cite{ColdingMinicozzi_ext}, and Morgan--Tian \cite{MorganTian_poinc}.
This rules out non-standard actions on 3--spheres 
and leads to a substantial strengthening 
of the conclusion of the previous corollary. 
 
\begin{cor}
\label{cor:uniquecomp2}
Suppose that the orientable closed 3--manifold $M_0$ 
is connected and irreducible.

(i) 
If $\pi_1(M_0)$ is finite, 
then the Ricci flow ${\mathcal M}$ 
goes extinct after finite time 
and $M_0$ is diffeomorphic to a spherical space form. 
The initial action $\rho_0\co G\acts M_0$ is standard. 

(ii)
If $\pi_1(M_0)$ is infinite, 
then the Ricci flow ${\mathcal M}$
does not go extinct after finite time. 
Every manifold $M_k$ has a unique 
connected component $M_k^{(0)}\cong M_0$, 
and the other components are $\cong S^3$. 
The action $\rho_k|_{M_k^{(0)}}\co G\acts M_k^{(0)}$ 
is smoothly conjugate to the initial action 
$\rho_0\co G\acts M_0$. 
\end{cor}
\begin{proof}
If $k_0<k_{\max}$,
then $M_0\cong {M_{k_0}^{(0)}}$ is a spherical space form 
by Corollary~\ref{cor:uniquecomp1}.
On the other hand, 
if $\pi_1(M_0)$ is finite,
then the Ricci flow ${\mathcal M}$ goes extinct in finite time 
and $k_0<k_{\max}$. 
Thus $k_0<k_{\max}$ if and only if $\pi_1(M_0)$ is finite
if and only if the Ricci flow ${\mathcal M}$ goes extinct in finite time. 

(i)
If $M_0\cong S^3$,
then Corollary~\ref{cor:uniquecomp1} (i)
and Proposition~\ref{decoirred} (i) 
imply that $\rho_{k-1}$ is standard if $\rho_k$ is standard. 
Moreover, 
$M_{k_{\max}}=\emptyset$ and thus $\rho_{k_{\max}}$ is standard.
It follows that $\rho_0$ is standard.

If $M_0$ is a spherical space form with nontrivial fundamental group,
then $0\leq k_0<k_{\max}$.
Since we now know that actions on unions of 3--spheres are standard,
Corollary~\ref{cor:uniquecomp1} (ii)
and Proposition~\ref{decoirred} yield
that $\rho_{k_0}$ is standard.
Furthermore, 
$\rho_k|_{M_k^{(0)}}$ is smoothly conjugate to 
$\rho_{k-1}|_{M_{k-1}^{(0)}}$ 
for $1\leq k\leq k_0$.
Hence $\rho_0$ is standard. 

(ii)
Now $k_0=k_{\max}$. 
As in case (i), 
Corollary~\ref{cor:uniquecomp1} (ii)  
and Proposition~\ref{decoirred} yield
that $\rho_k|_{M_k^{(0)}}$ is smoothly conjugate to 
$\rho_{k-1}|_{M_{k-1}^{(0)}}$ 
for $1\leq k\leq k_{\max}$. 
\end{proof}

\subsection{Applications to actions 
on elliptic and hyperbolic 3--manifolds}
\label{sec:appellhyp}

We prove now our main results. 

\begin{thmE}[Actions on elliptic manifolds are standard] 
Any smooth action by a finite group on an elliptic
3--manifold is smoothly conjugate to an isometric action. 
\end{thmE}

\begin{proof}
Let $\rho_0\co G\acts M_0$ 
be a smooth action by a finite group on a connected elliptic 3--manifold. 
We recall that elliptic 3--manifolds are orientable.
There exists an equivariant Ricci flow with cutoff
$\rho\co G\acts{\mathcal M}$
such that $\rho_0$ is the given action. 
(The $\rho_0$--invariant initial Riemannian metric $g_0$ 
is different from the a priori given spherical metric on $M_0$ 
unless the latter is already $\rho_0$--invariant 
in which case there is nothing to prove.) 
By Corollary~\ref{cor:uniquecomp2} (i),
the action $\rho_0$ is standard,
i.e.\ there exists a $\rho_0$--invariant spherical metric $g_{\sph}$ on $M_0$. 

As mentioned in section~\ref{defstan} already, 
any two diffeomorphic elliptic 3--manifolds are isometric.
Thus,
if $g'_{\sph}$ is an a priori given spherical metric on $M_0$,
then there exists a diffeomorphism $c$ of $M_0$ 
such that $c^*g_{\sph}=g'_{\sph}$. 
Then the action $c^*\rho_0$ is isometric with respect to $g'_{\sph}$.
\end{proof}

\begin{thmH}[Actions on closed hyperbolic manifolds are standard]
Any smooth action by a finite group on a closed hyperbolic 3--manifold
is smoothly conjugate to an isometric action.
\end{thmH}

\begin{proof}
Let $\rho_0\co G\acts M_0$ 
be a smooth action by a finite group 
on a closed connected hyperbolic 3--manifold. 

For a hyperbolic metric on $M_0$ 
there exists by Mostow rigidity \cite{Mostow} 
a unique isometric action $\rho_{\isom}\co G\acts M_0$
such that $\rho_0(\ga)$ is {\em homotopic} to $\rho_{\isom}(\ga)$ 
for every $\ga\in G$. 
We need to show that the actions $\rho_0$ and $\rho_{\isom}$ 
are smoothly {\em conjugate}.

We assume first that $M_0$ is orientable. 
There exists an equivariant Ricci flow with cutoff 
$\rho\co G\acts{\mathcal M}$
such that $\rho_0$ is the given action. 
By Corollary~\ref{cor:uniquecomp2} (ii),
${\mathcal M}$ does not go extinct in finite time, 
every manifold $M_k$ has a unique component $M_k^{(0)}\cong M_0$ 
and $\rho_k|_{M_k^{(0)}}$ 
is smoothly conjugate to $\rho_0$. 

We will now use the analysis of the 
long time behavior of the Ricci flow with cutoff, 
see \cite[Sections 87--92]{KL}. 
Since $M_0$ is not a graph manifold,
the thick parts of the time slices ${\mathcal M}_t$ 
cannot be empty for large $t$,
cf.\ \cite[Sections 89 and 92]{KL} and \cite[Theorem 0.2]{MorganTian_geom}. 
Therefore the collection of complete finite-volume hyperbolic 3--manifolds, 
which approximate the thick parts of the ${\mathcal M}_t$ 
as described in \cite[Proposition 90.1]{KL},
is nonempty. 
Furthermore, 
it can only consist of {\em one closed} connected hyperbolic 3--manifold $H$. 
This follows from the $\pi_1$--injectivity of the approximating maps,
compare \cite[Proposition 91.2]{KL}, 
and one uses, that 
$M_k-M_k^{(0)}$ is a union of 3--spheres 
and $M_k^{(0)}$ contains no incompressible 2--torus.

Let $T_0<+\infty$ 
and the nonincreasing function $\al\co [T_0,\infty)\to(0,\infty)$ 
with $\lim_{t\to\infty}\al(t)=0$ 
be as in \cite[Proposition 90.1]{KL}.
Since $H$ is closed, 
the conclusion of \cite[Proposition 90.1]{KL} yields: 
There exists $T_1\in[T_0,\infty)$
such that for any time $t\geq T_1$
there is an $\al(t)$--homothetic embedding 
\begin{equation*}
f_t\co H\to{\mathcal M}_t
\end{equation*}
which is a diffeomorphism 
onto a connected component of ${\mathcal M}_t$. 
Moreover, $f_t$ depends smoothly on $t$.
Note that the image of $f_t$ avoids the regions where surgeries take place,
because on it the scalar curvature is negative.  
Thus, 
if $t=t_k\geq T_1$ is a singular time,
then $\mathrm{im}(f_t)$ is a closed component of 
${\mathcal M}_t^-\cap{\mathcal M}_t^+$
which is not affected by surgeries.

All other components of the time slices ${\mathcal M}_t$, $t\geq T_1$,
are 3--spheres. 
These go extinct in finite time. 
Hence there exists $T_2\in[T_1,\infty)$
such that $f_t$ is a diffeomorphism onto ${\mathcal M}_t$ for $t\geq T_2$. 
We conclude that, up to scaling,
${\mathcal M}_t$ converges smoothly to $H$.
More precisely,
one has
\begin{equation*}
\frac{1}{2t}\,f_t^*g_t\buildrel{\mathcal C}^{\infty}\over\lra g_H
\end{equation*}
as $t\to\infty$,
compare part 1 of \cite[Proposition 90.1]{KL},
where we normalize the hyperbolic metric $g_H$ on $H$
to have sectional curvature $\equiv-1$. 

Let us denote $\rho_t:=\rho|_{{\mathcal M}_t}$. 
The pulled-back actions $f_t^*\rho_t$ on $H$, $t\geq T_2$, 
are smoothly conjugate to each other. 
Since they also become increasingly isometric
($\al(t)\to0$),
Arzel\`a-Ascoli implies that 
for any sequence $(t_n)$, $T_2\leq t_n\nearrow\infty$, 
the actions $f_{t_n}^*\rho_{t_n}$ 
subconverge smoothly to an isometric action $\bar\rho\co G\acts H$. 
(It must coincide with the unique isometric action 
homotopic to the $f_t^*\rho_t$ given by Mostow rigidity, 
compare our remark at the beginning of the proof.)
Thus for large $n$ the action $f_{t_n}^*\rho_{t_n}$ 
is a ${\cal C}^{\infty}$--small perturbation 
of the isometric limit action. 
Using the stability property of smooth actions 
(of compact Lie groups on closed manifolds) 
that sufficiently ${\cal C}^1$--small perturbations 
are smoothly conjugate (see Palais \cite{Palais} and Grove--Karcher
\cite{GroveKarcher}),  
it follows that the $f_t^*\rho_t$ are 
{\em smoothly conjugate} to $\bar\rho$ for $t\geq T_2$.
With Corollary~\ref{cor:uniquecomp2} (ii)
we conclude that 
$\rho_0\cong\rho_t\cong f_t^*\rho_t\cong\bar\rho$,
i.e.\ there exists a $\rho_0$--invariant hyperbolic metric $g_{\hyp}$ on $M_0$. 

Suppose now that $M_0$ is not orientable
and consider the orientable double covering $\hat M_0\to M_0$. 
Then $M_0$ is the quotient of $\hat M_0$ 
by a smooth orientation reversing involution $\iota$.
The action $\rho_0$ lifts to an action 
$\hat\rho_0\co \hat G\acts\hat M_0$ 
of an index two extension $\hat G$ of $G$. 
The nontrivial element in the kernel of the natural projection 
$\hat G\twoheadrightarrow G$ 
is $\iota$. 
By the above,
there exists a $\hat\rho_0$--invariant hyperbolic metric $\hat g_{\hyp}$
on $\hat M_0$. 
This metric descends to a $\rho_0$--invariant hyperbolic metric $g_{\hyp}$
on $M_0$. 
This finishes the proof that the action $\rho_0$ is geometric. 

Finally, 
if $g'_{\hyp}$ is an a priori given hyperbolic metric on $M_0$
(unrelated to the action $\rho_0$),
then Mostow rigidity yields a diffeomorphism $c$ of $M_0$ 
such that $c^*g_{\hyp}=g'_{\hyp}$. 
Thus the action $\rho_0$ can be smoothly conjugated 
to the $g'_{\hyp}$--isometric action $c^*\rho_0$. 
\end{proof}

The argument for proving Theorem H
can be extended to the case of 
actions on hyperbolic 3--manifolds with cusps. 
Let $M$ be a compact 3--manifold with nonempty boundary 
which admits a hyperbolic structure,
i.e.\ whose interior admits 
a complete hyperbolic metric $g_{\hyp}$ with finite volume. 
Then the ends of $(\Int(M),g_{\hyp})$ are cusps 
and the boundary components of $M$ are tori or Klein bottles. 

\begin{thm}[Actions on hyperbolic manifolds with cusps are standard]
\label{cuspcase}
Any smooth action $\rho\co G\acts M$ by a finite group 
is smoothly conjugate 
to the restriction of an isometric action on $(\Int(M),g_{\hyp})$ 
to a compact submanifold obtained from truncating the cusps. 
\end{thm}
\begin{proof}
As in the case of Theorem H
the non-orientable case can be reduced to the orientable one. 
We assume therefore that $M$ is orientable and connected. 

We denote by $\hat M_0$ the closed 3--manifold
obtained from doubling $M$ along the boundary. 
The action $\rho$ generates together with 
the natural involution $\iota\co \hat M_0\to\hat M_0$ 
the action $\hat\rho_0\co \hat G= G\times\Z/2\Z\acts\hat M_0$.
(We will identify $G$ with the subgroup $G\times\{0\}\subset\hat G$.) 

Let $\hat\rho\co \hat G\acts(\hat{\cal M},(\hat g_t))$ 
be an equivariant Ricci flow with cutoff 
whose initial action $\hat\rho_0$ is the given action. 
Since $\hat M_0$ is irreducible and has infinite fundamental group, 
the flow does not go extinct in finite time 
by Corollary~\ref{cor:uniquecomp2}. 
For each $t\geq0$ there is a unique connected component 
$\hat{\cal M}_t^{(0)}$ of $\hat{\cal M}_t^+$
(i.e.\ of $\hat{\cal M}_t$ if $t$ is a regular time)
which is 
$(\hat\rho_t,\hat\rho_0)$--equivariantly diffeomorphic to $\hat M_0$. 
We must take into account the possibility 
that surgeries occur arbitrarily late. 

Again $\hat M_0$ is not a graph manifold 
and hyperbolic components must form 
in the thick part of the time slice $\hat{\cal M}_t$. 
Since they are incompressible, 
they can only appear in the component $\hat{\cal M}_t^{(0)}\cong\hat M_0$. 

We consider one of these hyperbolic components.
To adapt the formulation of \cite[Proposition 90.1]{KL} to our purposes, 
we use the following notation:
Given a complete Riemannian manifold $N$ 
and a real number $r>0$, 
we denote by 
$N_r:=\{x\in N:\inj(x)\geq r\}$
its $r$--thick part in the injectivity radius sense. 
Then according to \cite[Proposition 90.1]{KL} there exist
a number $T_0\geq0$, 
a (continuous) nonincreasing function 
$\beta\co [T_0,\infty)\to(0,\infty)$ with 
$\lim_{t\to\infty}\beta(t)=0$, 
an orientable connected complete noncompact hyperbolic 3--manifold 
$(H,g_H)$ with finite volume,
and a smooth family of $\beta(t)$--homothetic embeddings
\begin{equation*}
f_t\co H_{\beta(t)}\embed \hat{\cal M}_t^{(0)},  \qquad t\geq T_0, 
\end{equation*}
avoiding the surgery regions and such that 
$\frac{1}{2t}\,f_t^*g_t\to {g_H}$ smoothly as $t\to\infty$. 
We may assume that $\beta$ is so small 
that $H_{\beta(t)}$ is a compact codimension-zero submanifold 
(a compact core) of $H$ bounded by horospherical 2--tori. 
(Note that the horospherical cross sections of the cusps of $H$ are tori 
because $H$ is orientable. 
Therefore the injectivity radius is constant along these cross sections.) 
Furthermore, the embeddings $f_t$ are $\pi_1$--injective. 

In view of the topological structure of $\hat M_0$, 
the uniqueness of the torus decomposition 
(see e.g.\ Hatcher \cite[Theorem 1.9]{Hatcher}) 
implies that $\mathrm{im}(f_t)$ corresponds to one of the hyperbolic components 
$M$ and $\iota M$ of $\hat M_0$ 
and is $G$--invariant up to isotopy. 
More precisely, 
for $0<\beta<\beta(T_0)$ and sufficiently large $t\geq T(\beta)$, 
the family of incompressible tori $f_t(\D H_{\beta})$ 
is isotopic to the family of ``separating'' tori 
$\Fix(\hat\rho_t(\iota)) \cap \hat{\cal M}_t^{(0)}$.
Moreover, 
for $\hat\ga\in\hat G$ 
the translate $\hat\rho_t(\hat\ga)(f_t(H_{\beta}))$ 
is isotopic to $f_t(H_{\beta})$ 
if and only if $\hat\ga\in G$. 
Here we use the fact that $M$ and $\iota M$ 
are not isotopic to each other in $\hat M_0$. 

We argue next that the almost hyperbolic components $f_t(H_{\beta})$ 
are essentially $G$--invariant and not only invariant up to isotopy. 

Let $0<\beta_1<\beta_2<\beta(T_0)$. 
Then for $\hat\ga\in\hat G$ and sufficiently large $t\geq T(\beta_1,\beta_2)$
the submanifold $\hat\rho_t(\hat\ga)(f_t(H_{\beta_2}))$ 
is disjoint from the collection of incompressible boundary tori 
$f_t(\D H_{\beta_1})$
because the injectivity radius is strictly smaller along the latter. 
Since $H_{\beta_2}$ is connected,
it follows that $\hat\rho_t(\hat\ga)(f_t(H_{\beta_2}))$ 
is either contained in or disjoint from $f_t(H_{\beta_1})$. 
In the first case $\hat\rho_t(\hat\ga)(f_t(H_{\beta_2}))$ 
is isotopic to $f_t(H_{\beta_1})$
and in the second case to its complement 
in $\hat{\cal M}_t^{(0)}$. 
It follows that 
$\hat\rho_t(\hat\ga)(f_t(H_{\beta_2}))\subset f_t(H_{\beta_1})$
if $\hat\ga\in G$
and 
$\hat\rho_t(\hat\ga)(f_t(H_{\beta_2}))\cap f_t(H_{\beta_2})=\emptyset$
otherwise. 
In particular, 
$\hat\rho_t(\hat\ga)(f_t(H_{\beta_2}))$ does not intersect 
the fixed point set of $\hat\rho_t(\iota)$ 
and is strictly contained in one of the halves of $\hat{\cal M}_t^{(0)}$ 
corresponding to $M$ and $\iota M$. 
As a consequence, 
we have for $0<\beta_1<\beta_2<\beta_3<\beta(T_0)$
and sufficiently large $t\geq T(\beta_1,\beta_2,\beta_3)$ 
that 
$f_t(H_{\beta_3})\subset\hat\rho_t(\ga)(f_t(H_{\beta_2})) 
\subset f_t(H_{\beta_1})$
for all $\ga\in G$. 
We see that, 
modulo the identification $\hat{\cal M}_t^{(0)}\cong\hat M_0$, 
there is an essentially $\rho$--invariant hyperbolic component 
forming inside $M$ and isotopic to $M$. 

This enables us to pull back the actions $\hat\rho_t|_G$ 
to almost isometric $G$--actions on large regions of $H$. 
Indeed, by the above 
there exists a sequence of times $t_n\nearrow\infty$ 
and $\hat\rho_{t_n}(G)$--invariant open subsets $U_n\subset \mathrm{im}(f_{t_n})$ 
with $f_{t_n}^{-1}(U_n)\nearrow H$. 
The pulled-back actions  
$f_{t_n}^*(\hat\rho_{t_n}|_G)$
are defined (at least) on $f_{t_n}^{-1}(U_n)$ 
and become more and more isometric. 
As in the proof of Theorem H, 
it follows using Arzel\`a-Ascoli 
that after passing to a subsequence
the actions 
$f_{t_n}^*(\hat\rho_{t_n}|_G)$ 
converge smoothly (and locally uniformly) to an isometric action 
$\bar\rho\co G\acts H$, 
i.e.\ for large $n$ the action 
$f_{t_n}^*(\hat\rho_{t_n}|_G)$ 
is a ${\cal C}^{\infty}$--small perturbation of $\bar\rho$. 
Applying Grove--Karcher \cite{GroveKarcher} again, 
there exist $f_{t_n}^*(\hat\rho_{t_n}|_G)$--invariant 
open subsets  
$V_n\subset f_{t_n}^{-1}(U_n)$ with $V_n\nearrow H$
and smooth embeddings 
$c_n\co V_n\to H$ with $c_n\to \id_H$ 
such that 
$\bar\rho(\ga)\circ c_n=c_n\circ f_{t_n}^*(\hat\rho_{t_n}(\ga))$
on $V_n$ for $\ga\in G$, 
compare the end of the proof of Lemma~\ref{equivarcaps}. 
In particular,
we obtain that for small $\beta\in(0,\beta(T_0))$ 
and sufficiently large $n\geq n(\beta)$ 
the $G$--actions 
$f_{t_n}^*\hat\rho_{t_n}|_{c_n^{-1}(H_{\beta})}$
and $\bar\rho|_{H_{\beta}}$ are smoothly conjugate. 
This implies that there exists a 
$(\bar\rho,\rho)$--equivariant embedding 
$\psi\co H_{\beta}\embed \Int(M)$. 

The complementary region $M-\psi(\Int(H_{\beta}))$ 
is a union of copies of $T^2\times[0,1]$. 
The restriction of the action $\rho$ to these components is standard 
according to Meeks and Scott \cite[Theorem 8.1]{MeeksScott}.
(A smooth finite group action on $\Si\times[0,1]$,
$\Si$ a closed surface,
is conjugate to a product action.)
We therefore can modify $\psi$ 
to be a $(\bar\rho,\rho)$--equivariant diffeomorphism 
$H_{\beta}\to \Int(M)$. 

The hyperbolic manifolds $(H,g_H)$ and $(\Int(M),g_{\hyp})$ 
are isometric
by Mostow-Prasad rigidity \cite{Mostow,Prasad}. 
This completes the proof of the theorem. 
\end{proof}

\subsection{Actions on $S^2\times\R$--manifolds} 
\label{sec:s2s1} 

Let $\rho\co G\acts{\mathcal M}$ be an equivariant Ricci flow with cutoff 
and suppose that $M_0$ is a closed connected 3--manifold 
admitting an $S^2\times\R$--structure,
that is a Riemannian metric locally isometric to $S^2\times\R$. 
Also in this case, the Ricci flow goes extinct in finite time 
for any initial metric, see Perelman
\cite{Perelman_extinction}, Colding--Minicozzi \cite{ColdingMinicozzi_ext}.

\begin{thm}\label{thm:s2s1}
The initial action $\rho_0\co G\acts M_0$ is standard.
\end{thm}
\begin{proof}
(i) Suppose that $M_0\cong S^2\times S^1$.

If ${\mathcal M}$ goes extinct at time $t_1$,
then $\rho_0$ is standard by Theorem~\ref{thm:ext}.
We assume therefore that 
${\mathcal M}$ does not go extinct at time $t_1$.
Since $S^2\times S^1$ is prime, 
and in particular does not have $\R P^3$ as a connected summand,
Theorem~\ref{thm:topeff} yields that 
$\rho_0$ is an equivariant connected sum of $\rho_1$,
respectively,
$\rho_1$ is an equivariant connected sum decomposition of $\rho_0$ 
in the sense of section~\ref{sec:eqcs}.

We first consider the case
when the family of 2--spheres, 
along which $M_0$ is decomposed,
contains a non-separating sphere.
Then all components of $M_1$ are 3--spheres and 
the action $\rho_1$ is standard by Theorem E. 
Using the notation of section~\ref{sec:eqcs},
the graph $\Ga$ associated to the connected sum decomposition of $M_0$ 
is homotopy equivalent to a circle. 
Equivalently, 
it contains a unique embedded cycle $\ga$
(which may be a loop) 
which consists precisely of the non-separating edges. 
We divide the family ${\mathcal P}$ of two point subsets 
into the subfamily ${\mathcal P}_1$ corresponding to the edges of $\ga$ 
and its complement ${\mathcal P}_2$. 
Let $M_1'$ be the union of the components of $M_1$ 
corresponding to the vertices of $\ga$.
Note that $M_1'$ and the ${\mathcal P}_i$ are $\rho_1$--invariant
due to the uniqueness of $\ga$. 
We have 
$\rho_0\cong(\rho_1)_{\mathcal P}
\cong((\rho_1)_{{\mathcal P}_2})_{{\mathcal P}_1}$. 
Lemma~\ref{trivsumm} yields as in the proof of 
Proposition~\ref{decoirred} that 
$(\rho_1)_{{\mathcal P}_2}\cong\rho_1|_{M_1'}$. 
So $\rho_0\cong(\rho_1|_{M_1'})_{{\mathcal P}_1}$. 
We may therefore assume without loss of generality 
that $\ga=\Ga$, $M_1'=M_1$ and ${\mathcal P}_1={\mathcal P}$. 
Let us denote the component 3--spheres of $M_1$
by $S^3_i$ parametrized by the index set $\Z/l\Z$, $l\geq1$, 
and the cyclic numbering chosen 
so that $x_i\in S^3_i$ and $y_i\in S^3_{i+1}$. 
We observe that $G_i$ must fix also $y_{i-1}$ and $x_{i+1}$. 
(Recall the definition of $G_i$ from the beginning of 
section~\ref{sec:eqcs}.) 
So $G_{i-1}\supseteq G_i\subseteq G_{i+1}$
and thus $G_1=\dots=G_l$. 
The $\rho_1$--invariant spherical metric on $M_1$ 
may be arranged so that $x_i$ and $y_{i-1}$ are antipodal for all $i$. 
When performing the equivariant connected sum operation 
to obtain $\rho_0$ from $\rho_1$  
we choose, in a $\rho_1$--equivariant way, 
Riemannian metrics on the $S_i^3-B_r(x_i)-B_r(y_i)$ 
isometric to $S^2\times[-1,1]$. 
This yields a $\rho_0$--invariant metric on $M_0$ 
locally isometric to $S^2\times\R$. 
Thus $\rho_0$ is standard in this case, too. 

Suppose now that the family of 2--spheres, 
along which $M_0$ is decomposed,
consists only of separating spheres. 
Then the graph $\Ga$ is a tree. 
Moreover, 
$M_1$ contains a unique component $M_1^{(0)}\cong S^2\times S^1$
and all other components are 3--spheres. 
The action on the union of 3--spheres is standard 
according to Theorem E. 
One proceeds as in the proof of Proposition~\ref{decoirred} (ii) 
and concludes that 
$\rho_1|_{M_1^{(0)}}$ is smoothly conjugate to $\rho_0$. 

Let $k\geq0$ be maximal such that $M_k$ contains a component 
$M_k^{(0)}\cong S^2\times S^1$. 
The above argument shows that ${\rho_k}|_{M_k^{(0)}}$ is standard
and, using induction, that $\rho_0\cong{\rho_k}|_{M_k^{(0)}}$. 
It follows that $\rho_0$ is standard. 

(ii) Now suppose that $M_0\not\cong S^2\times S^1$. 
In order to reduce to the $S^2\times S^1$--case
we argue as follows,
compare the end of the proof of Theorem H
in section~\ref{sec:appellhyp} above.

There exists a two-sheeted covering $\hat M_0\to M_0$
such that $\hat M_0\cong S^2\times S^1$. 
Indeed, 
inside the deck transformation group 
$\Deck(M_0)\subset \Isom(S^2\times\R)$ 
we consider the subgroup 
$\Deck(M_0)\cap \Isom_o(S^2\times\R)$ 
consisting of the deck transformations
which preserve the orientations of both the $S^2$-- and the $\R$--factor. 
Since the deck action $\Deck(M_0)\acts S^2\times\R$ is free,
this subgroup has index at most two 
and it is the deck transformation group $\Deck(\hat M_0)$ 
for a covering $\hat M_0\to M_0$. 
Clearly, $\hat M_0\cong S^2\times S^1$ and the covering is two-sheeted. 
(There are only three closed connected $S^2\times\R$--manifolds
besides $S^2\times S^1$, 
namely $\R P^3\sharp\,\R P^3$, $\R P^2\times S^1$
and the non-orientable $S^2$--bundle $S^2\tilde\times S^1$,
and these are double covered by $S^2\times S^1$, 
compare Scott \cite[Page 457]{Scott}.) 

Diffeomorphisms of $M_0$ lift to diffeomorphisms 
of the universal cover $S^2\times\R$ 
which normalize $\Deck(M_0)$. 
Let $\Diff_{++}(S^2\times\R)$ 
denote the index four normal subgroup of $\Diff(S^2\times\R)$ 
consisting of those diffeomorphisms 
which preserve the ends and act trivially on 
$H_2(S^2\times\R,\Z)\cong\Z$. 
Note that 
$\Isom_o(S^2\times\R)=\Isom(S^2\times\R)\cap \Diff_{++}(S^2\times\R)$. 
Thus the normalizer of $\Deck(M_0)$ in $\Diff(S^2\times\R)$
also normalizes 
$\Deck(\hat M_0)=\Deck(M_0)\cap \Diff_{++}(S^2\times\R)$. 
It follows that diffeomorphisms of $M_0$ 
lift to diffeomorphisms of $\hat M_0$, 
and the action $\rho_0$ lifts to an action 
$\hat\rho_0\co \hat G\acts\hat M_0$ of an index two extension 
$\hat G$ of $G$. 
By part (i) of the proof  
there exists a $\hat\rho_0$--invariant $S^2\times\R$--structure 
on $\hat M_0$.
It descends to a $\rho_0$--invariant $S^2\times\R$--structure on $M_0$. 
\end{proof}

\end{document}